\documentclass[a4paper,leqno,10pt]{amsart}

\usepackage{epsf,graphicx,epsfig}
\usepackage{amscd}
\usepackage{amsmath,latexsym,amssymb,amsthm}
\usepackage[nospace,noadjust]{cite}
\usepackage{textcomp}
\usepackage{setspace,cite}
\usepackage{lscape,fancyhdr,fancybox}
\usepackage{textcomp}
\usepackage{bm}
\usepackage{stmaryrd}
\usepackage[all,cmtip]{xy}
\usepackage{tikz}
\usetikzlibrary{shapes,arrows,decorations.markings}
\usepackage{graphicx}

\usepackage{color}
\usepackage{url}
\usepackage{enumerate}
\usepackage[mathscr]{euscript}
  \theoremstyle{plain}

\swapnumbers
    \newtheorem{thm}{Theorem}[section]
    \newtheorem{prop}[thm]{Proposition}
   \newtheorem{lemma}[thm]{Lemma}

    \newtheorem{subsec}[thm]{}

\theoremstyle{definition}
    \newtheorem{defn}[thm]{Definition}
    \newtheorem{remark}[thm]{Remark}
    \newtheorem{exam}[thm]{Example}

\theoremstyle{remark}

\title{}
\author{}
\date{}
\usepackage{amssymb}

\usepackage{hyperref}
\hypersetup{
	colorlinks,
	citecolor=blue,
	filecolor=black,
	linkcolor=blue,
	urlcolor=black
}

\title{Cohomology and deformations of dendriform algebras, and $\mathrm{Dend}_\infty$-algebras}

\author{Apurba Das}
\address{Department of Mathematics and Statistics,
Indian Institute of Technology, Kanpur 208016, Uttar Pradesh, India.}
\email{apurbadas348@gmail.com}


\begin{document}

\maketitle
\begin{abstract}
A dendriform algebra is an associative algebra whose product splits into two binary operations and the associativity splits into three new identities. These algebras arise naturally from some combinatorial objects and through Rota-Baxter operators. In this paper, we start by defining cohomology of dendriform algebras with coefficient in a representation. The deformation of a dendriform algebra $A$ is governed by the cohomology of $A$ with coefficient in itself.
Next we study $\mathrm{Dend}_\infty$-algebras (dendriform algebras up to homotopy) in which the dendriform identities hold up to certain homotopy.
% We show that a $\mathrm{Dend}_\infty$-algebra can be thought as an $A_\infty$-algebra whose $k$-ary operation splits into $k$ many operations and the $k$-th $A_\infty$-identity splits into $k$ many identities.
They are certain splitting of $A_\infty$-algebras. We define Rota-Baxter operator on $A_\infty$-algebras which naturally gives rise to $\mathrm{Dend}_\infty$-algebras.
Finally, we classify skeletal and strict $\mathrm{Dend}_\infty$-algebras.
\end{abstract}

\maketitle

%\tableofcontents

\section{Introduction}

In \cite{loday} Loday introduced a notion of diassociative algebra motivated from his study on periodicity phenomenons in algebraic $K$-theory. The Koszul dual (in the sense of Ginzburg and Kapranov \cite{ginz-kap}) of the operad of diassociative algebra is given by the operad of dendriform algebras. More precisely, a dendriform algebra is a vector space together with two binary operations satisfying three new identities. It turns out that the sum of the two binary operations is associative. Thus, a dendriform algebra can be thought as a splitting of an associative algebra. The notion of free dendriform algebra over a vector space was defined by Loday using planar binary trees. He also defined homology and cohomology theory of dendriform algebras with trivial coefficients \cite{loday}.

Rota-Baxter operator on associative algebras was introduced by G. Baxter in his study of probability theorey and was further studied by G.-C. Rota and his school \cite{baxter, rota}. The connection between Rota-Baxter operator and dendriform algebras was first observed by Aguiar who showed that a Rota-Baxter operator of weight zero carries a dendriform structure \cite{aguiar}. The adjoint functor of this construction was given by Ebrahimi-Farb and Guo by considering free Rota-Baxter algebras \cite{farb-guo}. See \cite{farb-man-pat, guba-kol, pei-bai-guo} for some other developments of dendriform algebras.

Despite the study of dendriform algebras, its connection with Rota-Baxter operator and various other algebras, the cohomology theory of dendriform algebras with coefficients and deformation theory has not yet properly studied. Some aspects of the deformation of dendriform algebras has been mentioned in \cite{leroux}, however, no cohomological interpretation has been obtained. On the other hand, while, strongly homotopy associative algebras ($A_\infty$-algebras) pay much attention from long years, (the 
split object)
strongly homotopy version of dendriform algebras has not yet studied. However, a definition of a strongly homotopy dendriform algebra has been mentioned in \cite{loday-val} from their study on the Koszul duality of operads. Our aim in this paper is to fulfill these gaps. We divide the results of the present paper into two parts which can be described as follows.

\medskip            

\subsection{Cohomology and deformations of dendriform algebras}
(Contents of section 2, 3)
We introduce a cohomology theory for dendriform algebras with coefficient in a representation. This cohomology theory is better understood in terms of certain combinatorial maps defined on the sets $\{ C_n | ~ n \geq 1 \}$, where $C_n$ is the set of first $n$ natural numbers. We show that the underlying cochain complex of a dendriform algebra with coefficient in itself carries a structure of a non-symmetric operad with a multiplication (Proposition \ref{dend-operad}, Remark \ref{dend-mul}). As a consequence, we get a Gerstenhaber structure on the cohomology with coefficients in itself (Remark \ref{gers-rem}). Some attempts in this direction has been done in \cite{yau}, however, there is some inaccuracy in the proof. 
We also relate the cohomology of a dendriform algebra with the Hochschild cohomology of the corresponding associative algebra (Theorem \ref{dend-hoch}). It is also shown that the second cohomology group can be represented by the equivalence classes of abelian extensions (Theorem \ref{2-coho-abel}).

Next, we study $1$-parameter formal deformation of dendriform algebras along the line of Gerstenhaber \cite{gers2}. The vanishing of the second cohomology of a dendriform algebra (with coefficient in itself) implies the rigidity of the dendriform structure. We also gave a universal deformation formula for dendriform algebras which can be thought as a splitting of the universal deformation formula for associative algebras constructed in \cite{gers18}. Finally, a deformation of finite order extends to a deformation of next order if the the third cohomology group vanishes (Theorem \ref{third-extension}).
% Motivated from the universal deformation formula for associative algebras , we give a universal deformation formular of dendriform algebras.

\medskip

\subsection{Strongly homotopy dendriform algebras}
(Contents of section 4, 5, 6)
The notion of $A_\infty$-space first arise in the work of Stasheff in the recognition of loop spaces \cite{stas}. The algebraic analogue of an $A_\infty$-space is called an $A_\infty$-algebra. Later on, $L_\infty$-algebras (strongly homotopy Lie algebras) was appeared in a supporting role in deformation theory \cite{lada-markl, lada-stas}. Various type of homotopy algebras are particular interest of current research.

In this paper, we concern about a strongly  homotopy version of dendriform algebras. More precisely, we introduce a notion of $\mathrm{Dend}_\infty$-algebra on a graded vector space $A$ in which the dendriform identities hold up to certain homotopy (Definition \ref{defn-dend-infty}). This notion also involves the combinatorial maps defined on $\{ C_n |~ n \geq 1 \}$. Our definition is equivalent to the one defined in \cite{loday-val}. We also define a notion of $\mathrm{Dend}_\infty[1]$-algebra which is equivalent to $\mathrm{Dend}_\infty$-algebra by a degree shift. Using this equivalent notion, we interpret a $\mathrm{Dend}_\infty$-algebra structure on $A$ as a square zero, degree $-1$ coderivation on the free diassociative coalgebra $\text{Diass}^{c}(V)$ on $V= sA$ (Theorem \ref{thm-dend-coder}). Thus, we obtain the following equivalent descriptions

\medskip

\begin{center}
$\mathrm{Dend}_\infty$-algebra on $A$ $\leftrightsquigarrow$ $\mathrm{Dend}_\infty[1]$-algebra on $V= sA$ $\leftrightsquigarrow$ $D \in \text{Coder}^{-1} (\text{Diass}^{c}(V))$ with $D^2 = 0$.
\end{center}

\medskip

\noindent A $\mathrm{Dend}_\infty$-algebra can be thought as splitting of an $A_\infty$-algebra whose $k$-ary operation splits into $k$ many operation and $k$-th identity splits into $k$ many new identities (Theorem \ref{split-a-inf}). 
We also define Rota-Baxter operator on $A_\infty$-algebras and show that they naturally give rise to  $\mathrm{Dend}_\infty$-algebras (Theorem \ref{rota-inf}). This generalizes the result of Aguiar to the homotopy context. Finally, following the result of Baez and Crans for $L_\infty$-algebras \cite{baez-crans}, we classify skeletal and strict $2$-term $\text{Dend}_\infty$-algebras (Theorems \ref{skeletal-2}, \ref{skeletal-n}, \ref{strict-cross}). More precisely, skeletal algebras correspond to third cohomology of dendriform algebras and strict algebras correspond to crossed module of dendriform algebras.

\medskip

Finally, in section 7 (Appendix), we recall some basics on operad with a multiplication, and $A_\infty$-algebras. All vector spaces, linear maps are over a field $\mathbb{K}$ of characteristic $0$.

\section{Cohomology of dendriform algebras}\label{sec-2}
Our aim in this section is to introduce and study cohomology of dendriform algebras with coefficient in a representation.

\begin{defn} \cite{loday}
A dendriform algebra is a vector space $A$ together with two bilinear maps $\prec, \succ : A \otimes A \rightarrow A$ satisfying
\begin{align}
& (a \prec b) \prec c = a \prec (b \prec c + b \succ c), \label{dend-eqn-1}\\
& (a \succ b) \prec c =  a \succ (b \prec c), \label{dend-eqn-2}\\
& (a \prec b + a \succ b) \succ c = a \succ (b \succ c), ~~~ \text{ for all } a, b , c \in A. \label{dend-eqn-3}
\end{align}
\end{defn}

A morphism between two dendriform algebras is a linear map which preserves the corresponding binary products.

Given a dendriform algebra $(A, \prec, \succ)$ one may define a new binary operation $\star$ on $A$ by
\begin{align*}
a \star b = a \prec b + a \succ b.
\end{align*}
Then it follows from (\ref{dend-eqn-1})-(\ref{dend-eqn-3}) that the product $\star$ is associative. Thus, a dendriform algebra can be thought as an associative algebra whose binary operation splits into two operations and whose associativity splits into three new identities.

Dendriform algebras arise naturally from Rota-Baxter operator on associative algebras. Let $(A, m )$ be an associative algebra and $R$ be a Rota-Baxter operator (of weight zero). In other words, $R : A \rightarrow A$ is a linear map satisfying 
\begin{align*}
m(R(a) , R(b)) = R \big( m(a , R(b)) + m(R(a) , b) \big), ~~~ \text{ for all } a, b \in A.
\end{align*}
In \cite{aguiar}, Aguiar showed that if $R$ is a Rota-Baxter operator on $(A, m)$, then the two multiplications $\prec$ and $\succ$ defined on $A$ by
\begin{align*}
a \prec b = m(a,  R (b)) ~~~~~~~~ \mathrm{ and } ~~~~~~~~ a \succ b = m(R(a), b)
\end{align*}
defines a dendriform structure on $A$. This result was extended by Ebrahimi-Farb who relates Rota-Baxter operator of arbitrary weight with dendriform trialgebra, a variant of dendrifom algebra introduced by Loday and Ronco \cite{farb, loday-ronco}.

The cohomology of a dendriform algebra with coefficients in $\mathbb{K}$ was introduced by Loday \cite{loday}. However, cohomology with coefficients in a representation is not yet defined. To do that, we need to introduce certain combinatorial maps which will be the key-point for the rest of the paper.

\medskip

Let $C_n = \{1, 2, \ldots, n \}$ be the set of first $n$ natural numbers. Since we will treat the elements of $C_n$ as certain symbols, we denote them by $\{ [1], [2], \ldots, [n] \}.$

For any $m , n \geq 1$ and $1 \leq i \leq m,$ we define maps $R_0 (m; \overbrace{1, \ldots, 1, \underbrace{n}_{i\text{-th place}}, 1, \ldots, 1}^{m}) : C_{m+n-1} \rightarrow C_m$ by

\begin{align*} R_0 (m; 1, \ldots, 1, n, 1, \ldots, 1) ([r]) ~=~
\begin{cases} [r] ~~~ &\text{ if } ~~ r \leq i-1 \\ [i] ~~~ &\text{ if } i \leq r \leq i +n -1 \\
[r -n + 1] ~~~ &\text{ if } i +n \leq r \leq m+n -1. \end{cases}
\end{align*}
We also define maps $R_i (m; \overbrace{1, \ldots, 1, \underbrace{n}_{i\text{-th place}}, 1, \ldots, 1}^{m}) : C_{m+n-1} \rightarrow \mathbb{K}[C_n]$ by

\begin{align*} R_i (m; 1, \ldots, 1, n, 1, \ldots, 1) ([r]) ~=~
\begin{cases} [1] + [2] + \cdots + [n] ~~~ &\text{ if } ~~ r \leq i-1 \\ [r - (i-1)] ~~~ &\text{ if } i \leq r \leq i +n -1 \\
[1]+ [2] + \cdots + [n] ~~~ &\text{ if } i +n \leq r \leq m+n -1. \end{cases}
\end{align*}

Let $A$ be any vector space. Consider the vector space
\begin{align*}
\mathcal{O}(n) := \mathrm{Hom}_{\mathbb{K}} (\mathbb{K}[C_n] \otimes A^{\otimes n}, A), ~~~~ \text{ for } n \geq 1.
\end{align*}
For $f \in \mathcal{O}(m), ~g \in \mathcal{O}(n)$ and $1 \leq i \leq m$, we define
$f \circ_i g \in \mathcal{O}(m+n-1)$ by
\begin{align*}
&(f \circ_i g) ([r]; a_1, \ldots, a_{m+n-1}) \\
&= f (R_0 (m; 1, \ldots, n, \ldots, 1)[r]; a_1, \ldots, a_{i-1}, g (R_i (m; 1, \ldots, n, \ldots, 1)[r]; a_i, \ldots, a_{i+n-1}), a_{i+n}, \ldots, a_{m+n-1})
\end{align*}
for $[r] \in C_{m+n-1}$ and $a_1, \ldots, a_{m+n-1} \in A$. Then we have the following.

\begin{prop}\label{dend-operad}
The partial compositions $\circ_i$ defines a non-$\sum$ operad structure on $\{ \mathcal{O}(n)|~ n \geq 1 \}$ with the identity element $\mathrm{id} \in \mathcal{O}(1)$ given by
\begin{align*}
\mathrm{id} ([1]; a) = a, ~~~ \mathrm{ for  } [1] \in C_1 ~~~ \mathrm{ and } ~~ a \in A.
\end{align*}
\end{prop}

\begin{proof}
Let $f \in \mathcal{O}(m)$, $g \in \mathcal{O}(n)$, $h \in \mathcal{O}(p)$ and $1 \leq i \leq m$, $1 \leq j \leq n$. Then for any $1 \leq r \leq i-1$, we have
\begin{align}\label{some-operad-iden1}
&((f \circ_i g ) \circ_{i+j-1} h ) ([r]; a_1, \ldots, a_{m + n + p - 2}) \nonumber \\
&= (f \circ_i g ) ([r]; a_1, \ldots , h ([1]+ \cdots + [p]; a_{i+j-1}, \ldots, a_{i+j+p-2}), a_{i+j+n-1}, \ldots, a_{m+n+p-2}) \nonumber \\
&= f \big( [r]; a_1, \ldots, g([1]+ \cdots + [n]; a_i, \ldots, h ([1]+ \cdots + [p]; a_{i+j-1}, \ldots, a_{i+j+p-2}), \ldots, a_{i+n+p-2} ),\\
& \hspace*{9cm} \ldots, a_{m+n+p-2} \big).  \nonumber
\end{align}
On the other hand,
\begin{align}\label{some-operad-iden}
&(f \circ_i (g \circ_j h )) ([r]; a_1, \ldots, a_{m+n+p-2}) \nonumber \\
& = f ([r]; a_1, \ldots, (g \circ_j h) ([1]+ \cdots + [n+p-1]; a_i, \ldots, a_{i+n+p-2}), \ldots, a_{m+n+p-2}).
\end{align}
Observe that
\begin{align*}
&(g \circ_j h) ([1]+ \cdots + [n+p-1]; a_i, \ldots, a_{i+n+p-2}) \\
&= (g \circ_j h) ([1]+ \cdots + [j-1];~ a_i, \ldots,  a_{i+n+p-2} )\\
&+ (g \circ_j h) ([j]+ \cdots + [j+p-1];~ a_i, \ldots, a_{i+n+p-2}) \\
&+ (g \circ_j h) ([j+p]+ \cdots + [n+p-1];~  a_i, \ldots, a_{i+n+p-2} ) \\
&= g \big([1]+ \cdots + [j-1] ;~ a_i, \ldots, h ([1]+ \cdots + [p];~ a_{i+j-1}, \ldots, a_{i+j+p-2}), \ldots, a_{i+n+p-2} \big)\\
&+ g \big([j] ;~ a_i, \ldots, h ([1]+ \cdots + [p];~ a_{i+j-1}, \ldots, a_{i+j+p-2}), \ldots, a_{i+n+p-2} \big)\\
&+ g \big([j+1]+ \cdots + [n];~  a_i, \ldots, h ([1]+ \cdots + [p];~ a_{i+j-1}, \ldots, a_{i+j+p-2}), \ldots, a_{i+n+p-2} \big)\\
&= g \big([1]+ \cdots + [n];~ a_i, \ldots, h ([1]+ \cdots + [p];~ a_{i+j-1}, \ldots, a_{i+j+p-2}), \ldots, a_{i+n+p-2} \big).
\end{align*}
Substituting this identity in (\ref{some-operad-iden}) and then comparing with (\ref{some-operad-iden1}) we get that
\begin{align}\label{operad-equa}
((f \circ_i g ) \circ_{i+j-1} h ) ([r]; a_1, \ldots, a_{m + n + p - 2}) = (f \circ_i (g \circ_j h )) ([r]; a_1, \ldots, a_{m+n+p-2}),
\end{align}
for $ r \leq i-1$. Similarly, one can show that (\ref{operad-equa}) holds if $r$ belongs to either of the interval $i \leq r \leq i+j-2$ / $i+j-1 \leq r \leq i+j+p-2$ / $i+j+p-1 \leq r \leq i+n+p-2$ / $i+n+p-1 \leq r \leq m+n+p-2$. Hence we have $((f \circ_i g ) \circ_{i+j-1} h ) = (f \circ_i (g \circ_j h ))$.

In a similar way, one can also show that $(f \circ_i g) \circ_{j+n-1} h = (f \circ_j h) \circ_i g$, for $1 \leq i < j \leq m$.

Finally, it is easy to see that the element $\mathrm{id} \in \mathcal{O}(1)$ is an identity element of the operad. Hence the proof. 
\end{proof}

\begin{remark}\label{dend-mul}
Let $(A, \prec, \succ)$ be a dendriform algebra. Define an element $\pi_A \in \mathcal{O}(2)$ by
$$ \pi_A ([r]; a, b) = \begin{cases}  a \prec b & \mbox{ if }  [r] = [1]\\
a \succ b & \mbox{ if } [r] = [2].
\end{cases} $$
It follows from the dendriform identities that $\pi_A $ satisfies $\pi_A \circ_1 \pi_A = \pi_A \circ_2 \pi_A$, or equivalently, $\pi_A \circ \pi_A = 0$ (see Appendix for the notation). Therefore, $\pi_A$
defines a multiplication in the operad considered in Proposition \ref{dend-operad}.

Thus, a linear map $f : A \rightarrow B$ between two dendriform algebras is a morphism if $\pi_A ([r]; a,  b ) = \pi_B ([r]; f(a), f(b))$, for all $[r] \in C_2$ and $a, b \in A.$
\end{remark}

\medskip

\subsection{Representation}

Let $(A , \prec, \succ)$ be a dendriform algebra. 
\begin{defn}
A representation of $A$ is given by a vector space $M$ together with two left actions 
\begin{align*}
\prec ~: A \otimes M \rightarrow M \qquad \succ ~: A \otimes M \rightarrow M
\end{align*}
and two right actions
\begin{align*}
\prec ~: M \otimes A \rightarrow M \qquad \succ ~: M \otimes A \rightarrow M
\end{align*}
satisfying the following $9$ identities
\begin{align*}
& (a \prec b) \prec m = a \prec (b \prec m + b \succ m),\\
& (a \succ b) \prec m =  a \succ (b \prec m),\\
& (a \prec b + a \succ b) \succ m = a \succ (b \succ m),\\
& \\
& (a \prec m) \prec c = a \prec (m \prec c + m \succ c),\\
& (a \succ m) \prec c =  a \succ (m \prec c),\\
& (a \prec m + a \succ m) \succ c = a \succ (m \succ c), \\
& \\
& (m \prec b) \prec c = m \prec (b \prec c + b \succ c),\\
& (m \succ b) \prec c =  m \succ (b \prec c),\\
& (m \prec b + m \succ b) \succ c = m \succ (b \succ c),
\end{align*}
for all $a, b, c \in A$ and $m \in M$.
\end{defn}

\begin{remark}\label{remark-repn}
To define the representation in a more compact form, we define two maps $\theta_1 : \mathbb{K}[C_2] \otimes (A \otimes M) \rightarrow M$ and $\theta_2 : \mathbb{K}[C_2] \otimes (M \otimes A) \rightarrow M$ by

$$ \theta_1 ([r]; a, m) = \begin{cases} a \prec m ~~~ &\mathrm{ if } ~~ [r]=[1] \\
a \succ m ~~~ &\mathrm{ if } ~~ [r] = [2] \end{cases} \qquad \quad
 \theta_2 ([r];  m, a) = \begin{cases} m \prec a ~~~ &\mathrm{ if } ~~ [r]=[1] \\
m \succ a ~~~ &\mathrm{ if } ~~ [r] = [2]. \end{cases}$$
Then the above $9$ identities of a representation can be expressed as
\begin{align*}
\theta_1 \big(  R_0 (2; 1, 2) [s]; ~ a,~ \theta_1 ( R_2 (2; 1, 2)[s]; b, m)  \big) =~& \theta_1 (R_0 (2; 2, 1)[s] ; ~\pi_A (R_1 (2; 2, 1)[s]; a, b),~ m), \\
\theta_1 \big(  R_0 (2; 1, 2) [s]; ~ a,~ \theta_2 ( R_2 (2; 1, 2)[s]; m, c)  \big) =~& \theta_2 (R_0 (2; 2, 1)[s] ; ~\theta_1 (R_1 (2; 2, 1)[s]; a, m),~ c),\\
\theta_2 \big(  R_0 (2; 1, 2) [s]; ~ m,~ \pi_A ( R_2 (2; 1, 2)[s]; b, c)  \big) =~& \theta_2 (R_0 (2; 2, 1)[s] ; ~\theta_2 (R_1 (2; 2, 1) [s] ; m, b),~ c), 
\end{align*}
for all $[s] \in C_3$ and $a, b, c \in A$, $m \in M$.
\end{remark}

Any dendriform algebra $A$ is a representation of itself with $\theta_1 = \theta_2 = \pi_A$. In such a case, all the above three identities are equivalent to $\pi_A \circ_2 \pi_A = \pi_A \circ_1 \pi_A$, which holds automatically as $\pi_A$ defines a multiplication.

\begin{prop}\label{prop-semi} (Semi-direct product) Let $(A, \prec, \succ)$ be a dendriform algebra and $M$ be a representation of it. Then the direct sum $A \oplus M$ inherits a dendriform algebra structure with the multiplications given by
\begin{align*}
(a, m) \prec (b, n) =~& (a \prec b,~ a \prec n + m \prec b),\\
(a, m) \succ (b, n) =~& (a \succ b,~ a \succ n + m \succ b),
\end{align*}
for $(a,m), (b,n) \in A \oplus M.$
\end{prop}

With the notations of Remark \ref{remark-repn}, the multiplication corresponding to the semi-direct product is given by
\begin{align*}
\pi_{A \oplus M} ([r]; (a,m), (b, n)) = \big( \pi_A ([r]; a, b) ,~ \theta_1 ([r]; a, n) + \theta_2 ([r]; m, b) \big), \text{ for } [r] \in C_2.
\end{align*}

\medskip

Let $(A, \prec, \succ)$ be a dendriform algebra and $M$ be a representation of it. We define the group $C^n_{\mathrm{dend}} (A, M)$ of $n$-cochains of $A$ with coefficients in $M$ by
\begin{align*}
C^n_{\mathrm{dend}} (A, M) := \mathrm{Hom}_\mathbb{K} (\mathbb{K}[C_n] \otimes A^{\otimes n}, M), ~~~ \text{ for } n \geq 1. 
\end{align*}
The coboundary is given by
\begin{align*}
&(\delta_{\mathrm{dend}} f) ([r]; a_1 , \ldots, a_{n+1}) \\
&=  \theta_1 \big( R_0 (2;1,n) [r]; ~a_1, f (R_2 (2;1,n)[r]; a_2, \ldots, a_{n+1})   \big) \\
&+ \sum_{i=1}^n  (-1)^i~ f \big(  R_0 (n; 1, \ldots, 2, \ldots, 1)[r]; a_1, \ldots, a_{i-1}, \pi_A (R_i (1, \ldots, 2, \ldots, 1)[r]; a_i, a_{i+1}), a_{i+2}, \ldots, a_{n+1}   \big) \\
&+ (-1)^{n+1} ~\theta_2 \big( R_0 (2; n, 1) [r];~ f (R_1 (2;n,1)[r]; a_1, \ldots, a_n), a_{n+1}   \big),
\end{align*}
for $[r] \in C_{n+1}$ and $a_1, \ldots, a_{n+1} \in A.$

The corresponding cohomology groups are denoted by $H^n_{\text{dend}} (A, M)$, for $n \geq 2$.
When we consider $\mathbb{K}$ as a representation of $A$ with trivial action ($\theta_1 = 0$ and $\theta_2 = 0$), the corresponding coboundary operator coincides with the one defined by Loday \cite{loday}, hence, our cohomology coincides with that of Loday for trivial coefficient.

\begin{remark}\label{gers-rem}
When $M = A$ with the representation given by $\theta_1 = \theta_2 = \pi_A$, the above coboundary map coincides with the one induced from the multiplication $\pi_A$. Therefore, the cohomology inherits a Gerstenhaber algebra structure. This first observation also ensures that $(\delta_{\text{dend}})^2 = 0$ for the coboundary map defined above with coefficients in an arbitrary representation.
\end{remark}

\begin{remark}
Let $(A, \prec, \succ)$ be a dendriform algebra and $M$ be a representation of it. Then $M$ can be considered as a representation of the corresponding associative algebra $(A, \star)$ via
\begin{align*}
a m = a \prec m + a \succ m \quad \text{ and } ~~ \quad ~~ m a = m \prec a + m \succ a.
\end{align*}
%We define a map $S : C^n_{\mathrm{dend}} (A, M) \rightarrow C^n_{\mathrm{Hoch}} (A, M)$ by $f \mapsto f_{n, [1]} + \cdots + f_{n, [n]}$.
%This map commute with the corresponding differentials. Hence we obtain a map $S_* : H^n_{\mathrm{dend}} (A, M) \rightarrow H^n_{\mathrm{Hoch}} (A, M).$
\end{remark}

The next result relates the cohomology of a dendriform algebra with the Hochschild cohomology of the corresponding associative algebra.

\begin{thm}\label{dend-hoch}
The map $S : C^n_{\mathrm{dend}} (A, M) \rightarrow C^n_{\mathrm{Hoch}} (A, M)$ defined by 
\begin{align*}
S(f) := f_{[1]} + \cdots + f_{[n]}
\end{align*}
commute with the differentials, i.e. $\delta_{\mathrm{Hoch}} \circ S = S \circ \delta_{\mathrm{dend}}$. Thus, it induces a map $S_* : H^n_{\mathrm{dend}} (A, M) \rightarrow H^n_{\mathrm{Hoch}} (A, M).$ Moreover, when $M = A$, the induced map $S_* : H^n_{\mathrm{dend}} (A, A) \rightarrow H^n_{\mathrm{Hoch}} (A, A)$ is in fact a morphism between Gerstenhaber algebras.
\end{thm}

\begin{proof}
The first part is a straightforward computation. We will prove the second part.

Consider the endomorphism operad $\mathrm{End}_A$ (see Appendix). Then similar to the first part, one can show that
\begin{align*}
S : \mathcal{O} (n) \rightarrow \mathrm{End}_A (n) , ~~~f \mapsto f_{[1]} + \cdots + f_{[n]}
\end{align*}
defines a morphism between operads. If $(A, \prec, \succ)$ is a dendriform algebra with the corresponding associative algebra $(A, \star)$, then the dendriform structure gives rise to a multiplication $\pi_A \in \mathcal{O} (2)$ and the associative structure gives rise to a multiplication (denote by the same symbol) $\star \in \mathrm{End}_A (2)$. It follows from the definition of $S$ that $S (\pi_A) = \star$. Thus, $S$ preserves the corresponding multiplications. As a remark, the induced map $S_* : H^n_{\mathrm{dend}} (A, A) \rightarrow H^n_{\mathrm{Hoch}} (A, A)$ is a morphism between Gerstenhaber algebras.
\end{proof}

\medskip

\subsection{Abelian extensions}

Let $(A, \prec, \succ)$ be a dendriform algebra and $M$ be a vector space. Note that $M$ can be considered as a dendriform algebra with the trivial multiplications.

\begin{defn}
An abelian extension of $A$ by $M$ is given by an extension
\[
\xymatrix{
0 \ar[r] & M \ar[r]^{i} & E \ar[r]^{j} & A \ar[r] & 0
}
\]
of dendriform algebras such that the sequence is split over $\mathbb{K}$.
\end{defn}

An abelian extension induces an $A$-representation on $M$ via the actions
\begin{align*}
\theta_1 ([r] ; a, m) =~& \pi_E ([r] ; s (a), i (m)),\\
\theta_2 ([r] ; m, a) =~& \pi_E ([r] ; i (m), s (a)),
\end{align*}
for $[r] \in C_2$, $a \in A$ and $m \in M$, where $s :A \rightarrow E$ be any section corresponding to the $\mathbb{K}$-splitting. Here $\pi_E : \mathbb{K} [C_2] \otimes E^{\otimes 2} \rightarrow E$ denotes the multiplication associated the dendriform structure on $E.$ One can easily verify that these actions are independent of the choice of $s$.

Two such abelian extensions are said to be equivalent if there is a morphism $\phi : E \rightarrow E'$ between dendriform algebras which makes the following diagram commute
\[
\xymatrix{
0 \ar[r] &  M \ar[r]^{i} \ar@{=}[d] & E \ar[d]^{\phi} \ar[r]^{j} & A \ar[r] \ar@{=}[d] & 0 \\
0 \ar[r] &  M \ar[r]_{i'} & E' \ar[r]_{j'} & A \ar[r]  & 0 .
}
\]

Now, fix an $A$-representation $M$. We denote by $\mathcal{E}xt (A, M)$ the equivalence classes of abelian extensions of $A$ by $M$ for which the induced representation on $M$ is the prescribed one. The next theorem relates $\mathcal{E}xt (A, M)$ with the second cohomology of the dendriform algebra $A$ with coefficients in $M$.

\begin{thm}\label{2-coho-abel}
There is a canonical bijection: $H^2_{\mathrm{dend}}(A, M) \cong \mathcal{E}xt ~(A, M).$
\end{thm}

\begin{proof}
Given a $2$-cocycle $f \in C^2_{\mathrm{dend}} (A, M)$, we consider the $\mathbb{K}$-module $E = A \oplus M$ with the following dendriform structure
\begin{align*}
\pi_E ([r]; (a, m), (b, n)) =~& \big( \pi_A ([r]; a, b), ~\theta_1 ([r]; a , n) + \theta_2 ( [r]; m, b) + f ([r]; a, b) \big).
\end{align*}
(Observe that when $f =0$ this is the semi-direct product.)
Using the fact that $f$ is a $2$-cocycle, it is easy to verify that $\pi_E$ defines a dendriform algebra structure on $E$. Moreover, $0 \rightarrow M \rightarrow E \rightarrow A \rightarrow 0$ defines an abelian extension with the obvious splitting. 
Let $\pi_E'$ be the multiplication on $E = A \oplus M$
associated to the cohomologous $2$-cocycle $f - \delta_{\mathrm{dend}} (g)$, for some $g \in C^1_{\mathrm{dend}} (A, M)$. The equivalence between abelian extensions $(E, \pi_E) $ and $(E, \pi_E')$ is given by $(a, m) \mapsto (a, m + g (a))$. Therefore, the map $H^2_{\mathrm{dend}} (A, M) \rightarrow \mathcal{E}xt (A, M) $ is well defined.

Conversely, given an extension 
$0 \rightarrow M \xrightarrow{i} E \xrightarrow{j} A \rightarrow 0$ with splitting $s$, we may consider $E = A \oplus M$ and $s$ is the map $s (a) = ( a, 0).$ With respect to the above splitting, the maps $i$ and $j$ are the obvious ones. Since $j \circ \pi_E ([r]; (a, 0), (b, 0)) = \pi_A ([r]; a, b)$ as $j$ is an algebra map, we have $\pi_E ([r]; (a, 0), ( b, 0)) = (f ([r]; a, b),~ \pi_A ([r]; a, b))$, for some $f \in C^2_{\mathrm{dend}} (A, M).$ 
Moreover, $\pi_E$ defines a dendriform algebra structure on $E$ implies that $f$ is a $2$-cocycle. Similarly, one can observe that any two equivalent extensions are related by a map $E = A \oplus M \xrightarrow{\phi} A \oplus M = E'$, $(a, m) \mapsto (a, m + g(a))$ for some $g \in C^1_{\mathrm{dend}} (A, M)$. Since $\phi$ is an algebra morphism, we have
\begin{align*}
\phi \circ \pi_E ([r]; ( a, 0), ( b, 0)) = \pi'_{E} ([r];~ \phi ( a, 0) , \phi (b, 0))
\end{align*}
which implies that $f' ([r]; a, b) = f ([r]; a, b) ~-~ (\delta_{\mathrm{dend}} g)([r]; a, b)$. Here $f'$ is the $2$-cocycle induced from the extension $E'$. This shows that the map $\mathcal{E}xt (A, M) \rightarrow H^2_{\mathrm{dend}} (A, M)$ is well defined. Moreover, these two maps are inverses to each other.
\end{proof}

\medskip

\section{Deformations}
In this section, we study deformations of dendriform algebras and provide cohomological interpretations. We also relate deformations of a dendriform algebra with deformations of the corresponding associative algebra.

\begin{defn}
Let $(A, \prec, \succ)$ be a dendriform algebra. A formal $1$-parameter deformation of $A$ is defined by two $\mathbb{K}[[t]]$-bilinear maps $\prec_t : A[[t]] \times A [[t]] \rightarrow A[[t]]$ and $\succ_t : A[[t]] \times A [[t]] \rightarrow A[[t]]$ that are of the form
\begin{align*}
a \prec_t b = \sum_{i \geq 0} (a \prec_i b )~t^i ~~~\text{ and }~~~ a\succ_t b= \sum_{i \geq 0} (a \succ_i b)~ t^i
\end{align*}
with $\prec_0 =~ \prec$ and $\succ_0 =~ \succ$ such that $(A[[t]], \prec_t,~ \succ_t)$ forms a dendriform algebra.
\end{defn}

That is, one must have the following identities
\begin{align*}
& (a \prec_t b ) \prec_t c = a \prec_t (b \prec_t c + b \succ_t c),\\
& (a \succ_t b) \prec_t c =  a \succ_t (b \prec_t c),\\
& (a \prec_t b + a \succ_t b) \succ_t c = a \succ_t (b \succ_t c)
\end{align*}
hold, for all $a, b, c \in A$. This is equivalent to following system of equations: for all $n \geq 0$,
\begin{align*}
&\sum_{i+j = n} \big(  (a \prec_i b ) \prec_j  - a \prec_i (b \prec_j c + b \succ_j c)  \big) = 0,\\
&\sum_{i+j = n} \big(  (a \succ_i b) \prec_j c -  a \succ_i (b \prec_j c)  \big) = 0,\\
&\sum_{i+j = n} \big(  (a \prec_i b + a \succ_i b) \succ_j c - a \succ_i (b \succ_j c)  \big) = 0,
\end{align*}
for all $a, b, c \in A.$

To write these system of equations in more compact form, we use the following notations. For each $i \geq 0$, define a map $\pi_i : \mathbb{K}[C_2] \otimes A^{\otimes 2} \rightarrow A$ by
$$\pi_i ([r]; a, b) = \begin{cases} a \prec_i b ~~~ &\mathrm{ if } ~~ [r]=[1] \\
a \succ_i b ~~~ &\mathrm{ if } ~~ [r] = [2]. \end{cases}$$
Thus, we have $\pi_0 = \pi_A$. With these notations, the above three system of identities can be expressed as
\begin{align}\label{def-sys}
\sum_{i+j = n} \pi_i \circ \pi_j = 0, ~~~ \text{ for all } n \geq 0.
\end{align}
For $n =1$, we get $\pi_A \circ \pi_1 + \pi_1 \circ \pi_A = 0$. This implies that the map $\pi_1 : \mathbb{K}[C_2] \otimes A^{\otimes 2}  \rightarrow A$ defines a $2$-cocycle of $A$ with coefficients in $A$. It is called the infinitesimal of the deformation. In particular, if $\pi_1 = \cdots = \pi_{n-1} = 0$, then $\pi_n$ defines a $2$-cocycle.

\begin{defn} Two deformations $(\prec_t, \succ_t)$ and $(\prec_t', \succ_t')$ of $A$ are said to be equivalent if there exists a formal map $\phi_t = \sum_{i \geq 0} \phi_i t^i$, where $\phi_i \in \text{Hom}_{\mathbb{K}} (A, A)$ with $\phi_0 = \text{id}_A$ such that
\begin{align*}
\phi_t (a \prec_t b) = \phi_t (a) \prec_t' \phi_t (b),   ~~~ \mathrm{ and } ~~~ \phi_t (a \succ_t b) = \phi_t (a) \succ_t' \phi_t (b).
\end{align*}
\end{defn}

Note that the conditions of the above definition are equivalent to
\begin{align*}
\sum_{i+j = n } \phi_i \circ \pi_j ([r]; a , b) = \sum_{i+j+k = n} \pi_k' ([r]; \phi_i (a), \phi_j (b)), ~~\text{ for } [r] \in C_2.
\end{align*}
Hence for $n = 1$, we get 
\begin{align*}
\pi_1 ([r]; a, b) + \phi_1 \circ \pi_A ([r]; a, b) = \pi_A ([r]; \phi_1 (a), b) + \pi_A ([r]; a, \phi_1 (b)) + \pi_1' ([r]; a, b),
\end{align*}
or, equivalently,
\begin{align*}
(\pi_1 - \pi_1) ([r]; a, b) =~&  \pi_A ([r]; \phi_1 (a), b) + \pi_A ([r]; a, \phi_1 (b)) -  \phi_1 \circ \pi_A ([r]; a, b) \\
=~&  ( \pi_A \circ \phi_1 - \phi_1 \circ \pi_A ) ([r]; a, b).
\end{align*}
Hence, $\pi_1 - \pi_1' = \delta_{\mathrm{dend}} (\phi_1)$. Therefore, the infinitesimals corresponding to equivalent deformations are cohomologous.

\medskip

Let $(A, m, R)$ be a Rota-Baxter algebra with the corresponding dendriform algebra $(A, \prec, \succ)$.
A deformation of $(A, m, R)$ is given by a deformation $m_t = \sum_{i \geq 0} m_i t^i$ of the associative algebra $A$ and a formal sum $R_t = \sum_{i \geq 0} R_i t^i$ where each $R_i \in \mathrm{End} (A, A)$ with $R_0 = R$ such that $(A[[t]], m_t, R_t)$ is a Rota-Baxter algebra. Then it follows from the result of Aguiar that $(A[[t]], \prec_{A[[t]]}, \succ_{A[[t]]})$ is a dendriform algebra. In other words, the triplet $(\prec_t, \succ_t)$ is a deformation of the corresponding dendriform algebra $(A, \prec, \succ)$ where
\begin{align*}
x \prec_t y =~& \sum_{n} \big( \sum_{i+j=n} m_i (x, R_j (y))\big) ~t^n, \\
x \succ_t y =~& \sum_{n} \big( \sum_{i+j=n} m_i (R_j (x), y) \big) ~t^n .
\end{align*}
%In a same way, the pair $(\prec_t', \succ_t')$ is a deformation of the corresponding dendriform dialgebra $(A, \prec', \succ')$ where
%\begin{align*}
%x \prec'_t y :=~& x \prec_t y  + x \cdot_t y = \big( \sum_{i+j=n} m_i (x, R_j (y)) + \lambda ~ m_n (x, y) ~\big) t^n, \\
%x \succ_t' y :=~& x \succ_t y = \sum_{i+j=n} m_i (R_j (x), y) t^n.
%\end{align*}

\begin{remark}
Let $(A, \prec, \succ)$ be a dendriform algebra and $(\prec_t, \succ_t)$ be a deformation of it.  Then it follows that $(A[[t]], \prec_t , ~\succ_t)$ is a dendriform algebra. Hence the pair $(A[[t]], \prec_t + \succ_t)$ is a deformation of the corresponding associative algebra $(A, \star = \prec + \succ).$
\end{remark}

\medskip

%\subsection{Rigidity}

\begin{defn}
A dendriform algebra $(A, \prec, \succ)$ is called rigid if every deformation of $A$ is equivalent to the trivial deformation $(\prec_t = \prec, ~ \succ_t = \succ)$.
\end{defn}

Let $(\prec_t, \succ_t)$ be a deformation of $A$. Then it is equivalent to some deformation
\begin{align*}
(\prec_t' =~ \prec + \sum_{i \geq k} \prec_i' , ~ \succ_t' =~ \succ + \sum_{i \geq k} \succ_i')
\end{align*}
in which the first non-zero term $\pi_k' = (\prec_k', \succ_k')$ is not a coboundary. This can be proved along the line of Gerstenhaber for associative algebras \cite{gers2}. Hence we obtain the following.

\begin{thm}
If $H^2_{\mathrm{dend}} (A, A) = 0$, then $A$ is rigid.
\end{thm}

\medskip

\subsection{An universal deformation}
Here we construct an universal deformation formula (UDF) for a dendriform algebra. We will show that this UDF for dendriform algebra can be seen as splitting of the UDF constructed by Gerstenhaber for associative algebras \cite{gers18}.

\begin{thm}\label{universal-thm}
Let $(A, \prec, \succ)$ be a dendriform algebra and $D_i : A \rightarrow A$ be $\mathbb{K}$-linear maps ($i=1,2$) satisfying
\begin{align*}
D_i (a \prec b ) = a \prec (D_i b) + (D_i a) \prec b,\\
D_i (a \succ b ) = a \succ (D_i b) + (D_i a) \succ b.
\end{align*}
Further, if $D_1 \circ D_2 = D_2 \circ D_1$, then
\begin{align*}
a \prec_t b =~& \sum_{n=0}^\infty \frac{t^n}{n!} ~ D_1^n (a) \prec D_2^n (b),\\
a \succ_t b =~& \sum_{n=0}^\infty \frac{t^n}{n!} ~ D_1^n (a) \succ D_2^n (b)
\end{align*}
defines a deformation of $A$.
\end{thm}

\begin{proof}
Define maps $\pi_i : \mathbb{K}[C_2] \otimes A^{\otimes 2} \rightarrow A$ by
\begin{align*}
\pi_i ([r]; a, b) := \frac{1}{i!} ~\pi_A ([r]; D_1^i (a), D_2^i (b)) = \frac{1}{i!} ~(D_1^i \cdot D_2^i)([r]; a, b),
\end{align*}
where the notation $D_1^i \cdot D_2^i$ has been defined in Appendix.
To prove that $(\prec_t, ~ \succ_t)$ is a deformation, we need to prove that $\sum_{i+j = n} \pi_i \circ \pi_j = 0$, or, equivalently, $ \sum_{i+j = n} \pi_i \circ_1 \pi_j = \sum_{i+j = n} \pi_i \circ_2 \pi_j$. First observe that
\begin{align}
\sum_{i+j = n} \pi_i \circ_1 \pi_j = \sum_{i=0}^n \pi_i \circ_1 \pi_{n-i} =~&  \sum_{i=0}^n \frac{1}{i! (n-i)!} ~ (D_1^i \circ (D_1^{n-i} \cdot {D_2}^{n-i})) \cdot {D_2}^i \nonumber \\
=~& \sum_{i=0}^n \frac{1}{i! (n-i)!} \sum_{j=0}^i {i \choose j}~ D_1^{n-i+j} \cdot (D_1^{i-j} \circ {D_2}^{n-i}) \cdot {D_2}^i  \nonumber \\
=~& \sum_{i,j=0, 1, \ldots, n, i \geq j} \frac{1}{(n-i)! j! (i-j)!}~ D_1^{n-j} \cdot (D_1^j \circ {D_2}^{n-i}) \cdot {D_2}^i. \label{univ-formula}
\end{align}
Similarly,
\begin{align*}
\sum_{i+j = n} \pi_i \circ_2 \pi_j = \sum_{i=0}^n \pi_i \circ_2 \pi_{n-i} =~&  \sum_{i=0}^n \frac{1}{i! (n-i)!}~ D_1^i  \cdot ({D_2}^i \circ (D_1^{n-i} \cdot {D_2}^{n-i})) \\
=~& \sum_{i=0}^n \frac{1}{i! (n-i)!} \sum_{j=0}^i {i \choose j} ~D_1^i \cdot ({D_2}^j \circ D_1^{n-i}) \cdot {D_2}^{n-j} \\
=~& \sum_{i,j=0, 1, \ldots, n, i \geq j} \frac{1}{(n-i)! j! (i-j)!} ~D_1^i \cdot ({D_2}^j \circ D_1^{n-i}) \cdot {D_2}^{n-j}.
\end{align*}
(The second last equality in each of the above two calculations follows from the fact that $\delta_{\text{dend}} (D_1) = \delta_{\text{dend}} (D_2) = 0$ \cite{gers18}.)
By replacing the dummy variables $i \leftrightarrow n-j$ and $j \leftrightarrow n-i$ and using the fact that $D_1 , {D_2}$ commute, we get the same expression as in  (\ref{univ-formula}). 
Hence the proof.
\end{proof}

\begin{remark}
Let $(A, \prec, \succ)$ be a dendriform algebra with the corresponding associative algebra $(A, \star)$. Let $D_1, D_2: A \rightarrow A$ be two $\mathbb{K}$-linear maps satisfying the conditions of Theorem \ref{universal-thm}. Then it follows that $D_1$ and $D_2$ are two commuting derivations of the associative algebra $(A, \star)$. Moreover, if we define
\begin{align*}
a \star_t b := a \prec_t b + a \succ_t b =  \sum_{n=0}^\infty \frac{t^n}{n!}~ D_1^n (a) \star D_2^n (b),
\end{align*}
then $\star_t$ is the UDF constructed by Gerstenhaber for associative algebras \cite{gers18}. Therefore, our UDF for dendriform algebras can be seen as a splitting of the Gerstenhaber's one.
\end{remark}

\medskip

\subsection{Extensions of deformations}
A deformation $(\prec_t, \succ_t)$ of $A$ is said to be of order $n$ if $\prec_t = \sum_{i=0}^n \prec_i t^i$ and $\succ_t = \sum_{i=0}^n \succ_i t^i$. In other words, $\pi_i = 0$, for all $i \geq n+1$. In the following, we assume that $H^2_{\text{dend}} (A, A) \neq 0$ so that we may obtain non-trivial deformations. Here we shall consider the problem of extending a deformation of order $n$ to a deformation of order $n+1$.

Suppose there are maps $\prec_{n+1}, ~ \succ_{n+1} : A^{\otimes 2} \rightarrow A$
such that $(\prec_t' =~ \prec_t + \prec_{n+1} t^{n+1},~ \succ_t' =~ \succ_t + \succ_{n+1} t^{n+1} )$ defines a deformation of $A$. Then we say that $(\prec_t, \succ_t)$ extends to a deformation of order $n+1$. As before, we define
 $\pi_{n+1} \in  \mathrm{Hom}_{\mathbb{K}} (\mathbb{K}[C_2] \otimes A^{\otimes 2}, A)$ by
 \begin{align*}
 \pi_{n+1} ([r]; a, b) = \begin{cases} a \prec_{n+1} b ~~~ &\mathrm{ if } ~~ [r]=[1] \\
a \succ_{n+1} b ~~~ &\mathrm{ if } ~~ [r] = [2]. \end{cases}
 \end{align*}

Since $(\prec_t, \succ_t)$ is a deformation of order $n$, it follows from (\ref{def-sys}) that
\begin{align*}
\pi_A \circ \pi_i + \pi_1 \circ \pi_{i-1} + \cdots + \pi_{i-1} \circ \pi_1 + \pi_i \circ \pi_A = 0, \text{ for } i=1, \ldots, n,
\end{align*}
or, equivalently, $\delta_{\mathrm{dend}} (\pi_i) = - \sum_{p+q = i, p, q \geq 1} \pi_p \circ \pi_q$. For $(\prec_t' =~ \prec_t + \prec_{n+1} t^{n+1}, \succ_t' =~ \succ_t + \succ_{n+1} t^{n+1} )$ to be a deformation, one more deformation equation need to satisfy, namely,
\begin{align*}
\delta_{\mathrm{dend}} (\pi_{n+1}) =  - \sum_{i+j =n +1, i, j \geq 1} \pi_i \circ \pi_j.
\end{align*}
The right hand side of the above equation is in $\mathrm{Hom}_{\mathbb{K}} (\mathbb{K}[C_3] \otimes A^{\otimes 3}, A)$ and it is called the obstruction to extend  the deformation $(\prec_t, \succ_t)$ to the next order.

\begin{lemma}
The obstruction is a $3$-cocycle, i.e
\begin{align*}
\delta_{\mathrm{dend}} (  - \sum_{i+j =n +1, i, j \geq 1} \pi_i \circ \pi_j ) = 0.
\end{align*}
\end{lemma}

\begin{proof}
For any $\pi, \pi' \in C^2_{\text{dend}} (A, A)$, it is easy to see that
\begin{align*}
\delta_{\text{dend}} (\pi \circ \pi') =  \pi \circ \delta_{\text{dend}} (\pi') - \delta_{\text{dend}} (\pi) \circ \pi' ~+~ \pi' \cdot \pi -~ \pi \cdot \pi'.
\end{align*}
(See \cite[Theorem 3]{gers} for the case of associative algebras.) Therefore,
\begin{align*}
\delta_{\text{dend}} \big(- \sum_{i+j =n+1, i, j \geq 1} \pi_i \circ \pi_{j} \big) =~& - \sum_{i+j =n+1, i, j \geq 1} \big( \pi_i \circ \delta_{\text{dend}} (\pi_j) - \delta_{\text{dend}} (\pi_i) \circ \pi_j \big) \\
=~& \sum_{p+q+r = n+1, p, q, r \geq 1} \big(  \pi_p \circ (\pi_q \circ \pi_r) - (\pi_p \circ \pi_q) \circ \pi_r \big)\\ =~& \sum_{p+q+r = n+1, p, q, r \geq 1} A_{p, q, r}    \qquad \mathrm{(say)}.
\end{align*}
The product $\circ$ is not associative, however, they satisfy the pre-Lie identity (see Appendix). This in particular implies that $A_{p, q, r} = 0$ whenever $q = r$. Finally, if $q \neq r$ then $A_{p, q, r} + A_{p, r, q} = 0$ by the pre-Lie identity. Hence we have $\sum_{p+q+r = n+1, p, q, r \geq 1} A_{p, q, r}   = 0$.
\end{proof}

Hence it follows that the obstruction defines a cohomology class in $H^3_{\mathrm{dend}}(A, A)$. If this cohomology class is zero, then the obstruction is given by a coboundary (say $\delta_{\text{dend}} (\pi_{n+1})$). In other words, $(\prec_t' = ~\prec_t + \prec_{n+1}t^{n+1},~ \succ_t' =~ \succ_t + \succ_{n+1} t^{n+1})$ defines a deformation of order $n+1$.

Thus we obtain the following.

\begin{thm}\label{third-extension}
Let $A$ be a dendriform algebra. If $H^3_{\mathrm{dend}} (A, A) = 0$, every deformation of finite order can be extended to a deformation of next order.
\end{thm}

\begin{remark}
Dendriform algebra is an algebra which can be expressed as a multiplication in a certain operad (Remark \ref{dend-mul}). There are other Loday-type algebras (di- and triassociative algebras, dendriform trialgebras, cubical algebras \cite{loday-ronco}) which also has the same feature \cite{yau}. In \cite{das} we study a formal deformation theory of multiplications in an operad, which gives deformation theory of all Loday-type algebras.
\end{remark}

\medskip

\section{Dendriform algebras up to homotopy}
In this section, we introduce a notion of $\mathrm{Dend}_\infty$-algebra (dendriform algebra up to homotopy) in which the dendriform identities hold up to certain homotopy. A coderivation interpretation of $\mathrm{Dend}_\infty$-algebra is also given. We also show that a $\mathrm{Dend}_\infty$-algebra can be thought as a splitting of $A_\infty$-algebra. 

\medskip

\subsection{$\mathrm{Dend}_\infty$-algebras}

\begin{defn}\label{defn-dend-infty}
A $\mathrm{Dend}_\infty$-algebra is a graded vector space $A = \oplus A_i$ together with a collection
\begin{align*}
\{ \mu_{k, [r]} : A^{\otimes k} \rightarrow A |~ [r] \in C_k, 1 \leq k < \infty \}
\end{align*}
of multilinear maps with $\mathrm{deg }(\mu_{k, [r]}) = k -2 $, satisfying the following identities
\begin{align}\label{dend-inf-iden}
\sum_{i+j = n+1} \sum_{\lambda = 1}^j (-1)^{\lambda (i+1) + i (|a_1| + \cdots + |a_{\lambda -1 }|)}~& \mu_{j, R_0 (j; 1, \ldots, i, \ldots, 1)[r]} \big( a_1, \ldots, a_{\lambda -1}, \\ &\mu_{i, R_\lambda (j; 1, \ldots, i, \ldots, 1)[r]} (a_\lambda, \ldots, a_{\lambda + i -1}), a_{\lambda +i}, \ldots, a_n \big) = 0, \nonumber
\end{align}
for any  $a_1, \ldots, a_n \in A$ and $[r] \in C_n$.
\end{defn}

Thus, a $\mathrm{Dend}_\infty$-algebra consists of $k$ many $k$-ary operations $\{ \mu_{k, [r]} : A ^{\otimes k} \rightarrow A | ~ [r] \in C_k \}$ labelled by the elements of $C_k$. These maps are supposed to satisfy the identities (\ref{dend-inf-iden}). Moreover, these identities are equivalent to the identities defined in \cite[Section 13.6]{loday-val} for a $\mathrm{Dend}_\infty$-algebra. Hence our definition is equivalent to the one defined in \cite{loday-val}.

One observes that for any fixed $n \geq 1$ and $a_1, \ldots, a_n \in A$, there are $n$ many identities between them associated to each $[r] \in C_n$. Let us explicitly describe them for small $n$. For $n =1$, we have that the degree $-1$ map $\mu_1 := \mu_{1, [1]}$ satisfies $(\mu_1)^2 = 0$. In other words, $(A, \mu_1)$ is a chain complex. For $n=2$ and $[r] = [1] \in C_2$, we get
\begin{align*}
\mu_1 (\mu_{2,[1]} (a,b)) = \mu_{2,[1]} (\mu_1 (a), b) + (-1)^{|a|}~ \mu_{2,[1]} (a, \mu_1 (b)).
\end{align*}
Similarly, for $[r] = [2] \in C_2$, we get
\begin{align*}
\mu_1 (\mu_{2,[2]} (a,b)) = \mu_{2,[2]} (\mu_1 (a), b) + (-1)^{|a|}~ \mu_{2,[2]} (a, \mu_1 (b)).
\end{align*}
This shows that the differential $\mu_1$ is a graded derivation for both the products $\mu_{2, [1]}$ and $\mu_{2, [2]}$. For $n = 3$, we have three relations corresponding to each elements of $C_3 = \{ [1], [2], [3] \}$. These relations are respectively given by

\begin{align*}
&\mu_{2,[1]} \big(\mu_{2,[1]} (a,b), c \big) - \mu_{2,[1]} \big( a,~ \mu_{2,[1]} (b, c) + \mu_{2,[2]} (b, c) \big) \\
&= \big\{ \mu_1 (\mu_{3,[1]} (a, b, c)) + \mu_{3,[1]} (\mu_1(a), b, c) + (-1)^{|a|}~ \mu_{3, [1]} (a, \mu_1 (b), c) + (-1)^{|a|+|b|}~ \mu_{3, [1]} (a, b, \mu_1 (c)) \big\},
\end{align*}

\begin{align*}
&\mu_{2,[1]} (\mu_{2,[2]} (a,b) ,~ c) - \mu_{2,[2]} (a,~ \mu_{2,[1]} (b, c)) \\
&= \big\{ \mu_1 (\mu_{3,[2]} (a, b, c)) + \mu_{3,[2]} (\mu_1(a), b, c) + (-1)^{|a|}~ \mu_{3, [2]} (a, \mu_1 (b), c) + (-1)^{|a|+|b|}~ \mu_{3, [2]} (a, b, \mu_1 (c)) \big\},
\end{align*}

\noindent and

\begin{align*}
&\mu_{2,[2]} (\mu_{2,[1]} (a,b) + \mu_{2, [2]} (a,b) ,~ c ) - \mu_{2,[2]} (a,~ \mu_{2,[2]} (b, c)) \\
&= \big\{ \mu_1 (\mu_{3,[3]} (a, b, c)) + \mu_{3,[3]} (\mu_1(a), b, c) + (-1)^{|a|}~ \mu_{3, [3]} (a, \mu_1 (b), c) + (-1)^{|a|+|b|}~ \mu_{3, [3]} (a, b, \mu_1 (c)) \big\}.
\end{align*}

This shows that the degree $0$ products $\mu_{2, [1]}$ and $\mu_{2, [2]}$ does not in general satisfy the dendriform identities. However, they  satisfy the same identities up to some terms involving $\mu_{3}$'s. Similarly, for higher $n$, we get higher coherence laws that $\mu_{k,[r]}$'s must satisfy.

Any dendriform algebra is a $\text{Dend}_\infty$-algebra concentrated in degree zero. Differential graded dendriform algebras are also examples of $\text{Dend}_\infty$-algebras in which $\mu_{k, [r]} = 0$, for $k \geq 3$. These algebras arise from Rota-Baxter operator on differential graded associative algebras. More generally, in section 5, we define Rota-Baxter operator on $A_\infty$-algebras which gives rise to $\text{Dend}_\infty$-algebras.

Any $A_\infty$-algebra $(A, \mu_k)$ can also be thought as a $\text{Dend}_\infty$-algebra in which $\mu_{k, [1]}= \mu_k$ and $\mu_{k, [r]} = 0$, for $r > 1$ (alternatively, $\mu_{k, [k]} = \mu_k$ and $\mu_{k, [r]} = 0$, for $r < k$).

The direct sum of two $\text{Dend}_\infty$-algebras is again a algebra of same type. Some other examples of $\text{Dend}_\infty$-algebras can be constructed from the results of section 6 (see also Examples \ref{exm-str1}, \ref{exm-str2}).

\medskip

An equivalent definition of a $\mathrm{Dend}_\infty$-algebra can be given by a degree shift.  This definition is essential to derive a coderivation interpretation of a $\text{Dend}_\infty$-algebra.
\begin{defn}
A $\mathrm{Dend}_\infty [1]$-algebra is a graded vector space $V$ together with a collection of degree $-1$ maps $\{ \varrho_{k, [r]} : V^{\otimes k} \rightarrow V |~ [r] \in C_k, 1 \leq k < \infty \}$
satisfying the following identities:
\begin{align*}
\sum_{i+j = n+1} \sum_{\lambda = 1}^j (-1)^{ |v_1| + \cdots + |v_{\lambda -1 }|}~ \varrho_{j , R_0 (j; 1, \ldots, i, \ldots, 1)[r]} (v_1, \ldots, v_{\lambda -1}, \varrho_{i, R_\lambda (j; 1, \ldots, i, \ldots, 1)[r]} (v_\lambda, \ldots, v_{\lambda + i -1}),\\
 v_{\lambda +i}, \ldots, v_n) = 0,
\end{align*}
for all $v_1, \ldots, v_n \in V$ and $[r] \in C_n$.
\end{defn}

Let $(A, \mu_{k, [r]})$ be a $\mathrm{Dend}_\infty$-algebra. Take $V = sA$, where $V_i = (sA)_i = A_{i-1}$. For each $1 \leq k < \infty$ and $[r] \in C_k$, we define
\begin{align*}
\varrho_{k, [r]} = (-1)^{\frac{k(k-1)}{2}}~ s \circ \mu_{k, [r]} \circ (s^{-1})^{\otimes k},
\end{align*}
where $s : A \rightarrow sA$ is the degree $+1$ map and $s^{-1} : sA \rightarrow A$ is the degree $-1$ map. It turns out that $\mathrm{deg} (\mu_{k, [r]}) = -1$. Note that $\mu_{k, [r]}$ can be constructed from $\varrho_{k, [r]}$ by $\mu_{k, [r]} = s^{-1} \circ \varrho_{k, [r]} \circ  s^{\otimes k}$. The sign $(-1)^{\frac{k(k-1)}{2}}$ is a consequence of the Koszul sign convension in the graded context.
Thus, we have the following. The proof is similar to case of $A_\infty$-algebra.

\begin{prop}\label{prop-dend-dend1}
A pair $(A, \mu_{k, [r]})$ is a $\mathrm{Dend}_\infty$-algebra if and only if $(V, \varrho_{k, [r]})$ is a $\mathrm{Dend}_\infty [1]$-algebra.
\end{prop}

\medskip

\subsection{Coderivation interpretation}
We give a new interpretation of a $\mathrm{Dend}_\infty$-algebra structure on a graded vector space $A$ in terms of a coderivation on the graded diassociative coalgebra on the suspension of $A$. We begin with the following.
\begin{defn}
A (graded) diassociative coalgebra is a (graded) vector space $C$ together with degree zero linear maps $\triangle_1, \triangle_2 : C \rightarrow C \otimes C$ satisfying
\begin{itemize}
\item $(\mathrm{id} \otimes \triangle_1) \triangle_1 = (\triangle_1 \otimes \mathrm{id}) \triangle_1 = (\mathrm{id} \otimes \triangle_2) \triangle_1$,
\item $(\triangle_2 \otimes \mathrm{id}) \triangle_1 = (\mathrm{id} \otimes \triangle_1) \triangle_2$,
\item $(\triangle_1 \otimes \mathrm{id}) \triangle_2 = (\mathrm{id} \otimes \triangle_2) \triangle_2 = (\triangle_2 \otimes \mathrm{id}) \triangle_2.$
\end{itemize}
\end{defn}

This notion is dual to the notion of diassociative algebra (dialgebra) introduced by Loday \cite{loday}. Moreover, it follows from the above definition that $\triangle_1 , \triangle_2$ are both associative comultiplications on $C$ satisfying additional three identities.

A coderivation on a (graded) diassociative coalgebra $(C, \triangle_1, \triangle_2)$ is a linear map $d : C \rightarrow C$ which is a coderivation for both the coproducts, i.e
\begin{align*}
\triangle_1 \circ d = (\mathrm{id} \otimes d + d \otimes \mathrm{id}) \circ \triangle_1 \quad \text{ and } \quad
\triangle_2 \circ d = (\mathrm{id} \otimes d + d \otimes \mathrm{id}) \circ \triangle_2.
\end{align*}

The set of all coderivations on $C$ is denoted by $\mathrm{Coder} (C)$.

\medskip

Let $V$ be a graded vector space. Consider the free diassociative coalgebra over $V$ given by $\mathrm{Diass}^c (V) = TV \otimes V \otimes TV$ with the coproducts

\begin{align*}
&\triangle_1 (v_1 \cdots v_n \otimes v \otimes w_1 \cdots w_m) \\
& = \sum_{i, j \geq 0, i+j +1 \leq m}  (v_1 \cdots v_n \otimes v \otimes w_1 \cdots w_i ) \otimes (w_{i+1} \cdots w_{i+j} \otimes w_{i+j+1} \otimes w_{i+j+2} \cdots w_m), \\\\
& \triangle_2 (v_1 \cdots v_n \otimes v \otimes w_1 \cdots w_m) \\
& = \sum_{i, j \geq 0, i+j+1 \leq n} (v_1 \cdots v_i \otimes v_{i+1} \otimes v_{i+2} \cdots v_{i+j+1}) \otimes (v_{i+j+2} \cdots v_n \otimes v \otimes w_1 \cdots w_m).
\end{align*}

Let $\varrho_{k, [1]}, \ldots, \varrho_{k, [k]} : V^{\otimes k} \rightarrow V$ be a sequence of $k$ many $k$-ary multilinear maps (of degree $s$). Consider a linear map $\varrho_k : TV \otimes V \otimes TV \rightarrow V$ which is non-zero only on $\oplus_{i+j = k+1, i, j \geq 1} V^{\otimes i-1} \otimes V \otimes V^{\otimes j-1}$ and is given by
\begin{align*}
\varrho_k (1 \otimes v_1 \otimes v_2 \cdots v_k ) =~& \varrho_{k, [1]} (v_1, \ldots, v_k), \\
\varrho_k (v_1 \otimes v_2 \otimes v_3 \cdots v_k) =~& \varrho_{k, [2]} (v_1, \ldots, v_k), \\
 \vdots ~&  \\
\varrho_k (v_1 \cdots v_{k-1} \otimes v_k \otimes 1) =~& \varrho_{k, [k]}  (v_1, \ldots, v_k).
\end{align*}

Then $\varrho_k$ extends to a coderivation $\widetilde{\varrho_k} : TV \otimes V \otimes TV \rightarrow TV \otimes V \otimes TV$ of degree $s$ of the free diassociative coalgebra $\mathrm{Diass}^c (V) = TV \otimes V \otimes TV$. More precisely, $\widetilde{\varrho_k}$ is given by
\begin{align*}
&\widetilde{\varrho_k} (v_1 \cdots v_n \otimes v \otimes w_1 \cdots w_m)  \\
& = \sum_{(1)} (-1)^{|v_1| + \cdots + |v_l|}~ v_1 \ldots v_l \varrho_k (v_{l+1} \cdots v_{l + i-1} \otimes v_{l+i} \otimes v_{l+i+1} \cdots v_{l+i+j-1}) \cdots v_n \otimes v \otimes w_1 \cdots w_m \\
& + \sum_{(2)} (-1)^{|v_1| + \cdots + |v_{n-i+1}|}~ v_1 \cdots v_{n-i+1} \otimes \varrho_k (v_{n-i+2} \cdots v_n \otimes v \otimes w_1 \cdots w_{j-1}) \otimes w_j \cdots w_m\\
& + \sum_{(3)} (-1)^{|v_1| + \cdots + |w_l|}~ v_1 \cdots v_n  \otimes v \otimes w_1 \cdots w_l \varrho_k (w_{l+1} \cdots w_{l + i-1} \otimes w_{l+i} \otimes w_{l+i+1} \cdots w_{l+i+j-1}) \cdots w_n,
\end{align*}
where the index sets are given by
\begin{align*}
(1) =~& \{ l \geq 0;~ i, j \geq 1 \text{ with } i+j = k+1;~ l+i+j-1 \leq n \},\\
(2) =~& \{ n-i+1 \geq 0;~ i, j \geq 1 \text{ with } i+j = k+1;~ j-1 \leq m \},\\
(3) =~& \{ l \geq 0; ~ i, j \geq 1 \text{ with } i+j = k+1;~ l+i+j-1 \leq m \}.
\end{align*}

The following lemma is useful to prove the next theorem.
\begin{lemma}\label{coder-lemma}
For any fixed $i, j \geq 1$ with $i+j = n+1$, and $1 \leq r \leq n,$ we have
\begin{align*}
&\widetilde{\varrho_j} \circ \widetilde{\varrho_i} (v_1 \cdots v_{r-1} \otimes v_r \otimes v_{r+1} \cdots v_n) \\
&= \sum_{\lambda = 1}^j (-1)^{|v_1| + \cdots + |v_{\lambda -1}|} ~ \varrho_{j, R_0 (j; 1, \ldots, i, \ldots, 1) [r]} (v_1, \ldots, v_{\lambda -1}, \varrho_{i, R_\lambda (j; 1, \ldots, i, \ldots, 1) [r]} (v_\lambda, \ldots, v_{\lambda +i -1}) , v_{\lambda +i}, \ldots, v_n).
\end{align*}
\end{lemma}

\begin{proof}
We start with the right hand side. Observe that
\begin{align*}
\sum_{\lambda =1}^j = \sum_{\lambda \leq r-i} + \sum_{r - i +1 \leq \lambda \leq r} +  \sum_{\lambda -1 \geq r}.
\end{align*} 
Note that
\begin{align}\label{ident-1}
&\sum_{\lambda \leq r-i} \pm ~\varrho_{j, R_0 (j; 1, \ldots, i, \ldots, 1) [r]} (v_1, \ldots, v_{\lambda -1}, \varrho_{i, R_\lambda (j; 1, \ldots, i, \ldots, 1) [r]} (v_\lambda, \ldots, v_{\lambda +i -1}) , v_{\lambda +i}, \ldots, v_n) \nonumber \\
&= \sum_{\lambda \leq r-i} \pm ~\varrho_{j, [r-i+1]} (v_1, \ldots, v_{\lambda -1}, \varrho_{i,[1]+ \cdots + [i]} (v_\lambda, \ldots, v_{\lambda +i -1}) , v_{\lambda +i}, \ldots, v_n) \nonumber \\
&= \sum_{\lambda \leq r-i} \pm ~\varrho_{j} (v_1 \cdots v_{\lambda -1} \varrho_{i,[1]+ \cdots + [i]} (v_\lambda, \ldots, v_{\lambda +i -1})  v_{\lambda +i} \cdots \otimes v_r \otimes v_{r+1} \cdots v_n).
\end{align}
Similarly,
\begin{align}\label{ident-2}
&\sum_{r-i+1 \leq \lambda \leq r} \pm~ \varrho_{j, R_0 (j; 1, \ldots, i, \ldots, 1) [r]} (v_1, \ldots, v_{\lambda -1}, \varrho_{i, R_\lambda (j; 1, \ldots, i, \ldots, 1) [r]} (v_\lambda, \ldots, v_{\lambda +i -1}) , v_{\lambda +i}, \ldots, v_n) \nonumber \\
&= \sum_{r-i+1 \leq \lambda \leq r} \pm~ \varrho_{j, [\lambda]} (v_1, \ldots, v_{\lambda -1}, \varrho_{i, [r-(\lambda -1)]} (v_\lambda, \ldots, v_{\lambda +i -1}) , v_{\lambda +i}, \ldots, v_n) \nonumber \\
&=  \sum_{r-i+1 \leq \lambda \leq r} \pm~ \varrho_{j} (v_1 \cdots v_{\lambda -1} \otimes \varrho_{i} (v_\lambda \cdots v_{r-1} \otimes v_r \otimes v_{r+1} \cdots v_{\lambda +i -1}) \otimes v_{\lambda +i} \cdots v_n)
\end{align}
and
\begin{align}\label{ident-3}
&\sum_{\lambda -1 \geq r} \pm ~ \varrho_{j, R_0 (j; 1, \ldots, i, \ldots, 1) [r]} (v_1, \ldots, v_{\lambda -1}, \varrho_{i, R_\lambda (j; 1, \ldots, i, \ldots, 1) [r]} (v_\lambda, \ldots, v_{\lambda +i -1}) , v_{\lambda +i}, \ldots, v_n) \nonumber \\
&= \sum_{\lambda -1 \geq r} \pm ~ \varrho_{j, [r]} (v_1, \ldots, v_{\lambda -1}, \varrho_{i,[1]+ \cdots [i]} (v_\lambda, \ldots, v_{\lambda +i -1}) , v_{\lambda +i}, \ldots, v_n) \nonumber \\
&= \sum_{\lambda -1 \geq r} \pm ~ \varrho_{j} (v_1 \cdots v_{r-1} \otimes v_r \otimes \cdots v_{\lambda -1} \varrho_{i,[1]+ \cdots [i]} (v_\lambda, \ldots, v_{\lambda +i -1}) v_{\lambda +i} \cdots v_n).
\end{align}
Hence, by adding (\ref{ident-1}), (\ref{ident-2}), (\ref{ident-3}) and using the definition of $\widetilde{\varrho_i}$, we get
\begin{align*}
&\sum_{\lambda = 1}^j \pm ~\varrho_{j, R_0 (j; 1, \ldots, i, \ldots, 1) [r]} (v_1, \ldots, v_{\lambda -1}, \varrho_{i, R_\lambda (j; 1, \ldots, i, \ldots, 1) [r]} (v_\lambda, \ldots, v_{\lambda +i -1}) , v_{\lambda +i}, \ldots, v_n)\\
&= \varrho_j \circ \widetilde{\varrho_i} (v_1 \cdots v_{r-1} \otimes v_r \otimes v_{r+1} \cdots v_n) \\
&= \widetilde{\varrho_j} \circ \widetilde{\varrho_i} (v_1 \cdots v_{r-1} \otimes v_r \otimes v_{r+1} \cdots v_n)  \quad (\text{since } i+j = n+1).
\end{align*}
\end{proof}

\begin{thm}\label{thm-dend1-coder}
Let $V$ be a graded vector space and $\{ \varrho_{k, [r]} : V^{\otimes k} \rightarrow V |~ [r] \in C_k, 1 \leq k < \infty \}$ be a sequence of maps of degree $-1$. Consider the coderivation $D = \sum_{k \geq 1} \widetilde{\varrho_k} \in \mathrm{Coder}^{-1} (\mathrm{Diass}^c (V))$. Then $(V, \varrho_{k, [r]})$ is a $\mathrm{Dend}_\infty [1]$-algebra if and only if $D \circ D = 0$.
\end{thm}

\begin{proof}
Note that, the condition $D \circ D = 0$ is same as
\begin{align*}
(\widetilde{\varrho_1} +  \widetilde{\varrho_2} + \widetilde{\varrho_3} + \cdots ) \circ (\widetilde{\varrho_1} +  \widetilde{\varrho_2} + \widetilde{\varrho_3} + \cdots ) = 0.
\end{align*}
This is equivalent to a system of relations: for all $n \geq 1$,
\begin{align*}
\widetilde{\varrho_1} \circ \widetilde{\varrho_n} + \widetilde{\varrho_2} \circ \widetilde{\varrho_{n-1}} + \cdots + \widetilde{\varrho_{n-1}} \circ \widetilde{\varrho_2} + \widetilde{\varrho_n} \circ \widetilde{\varrho_1} = 0,
\end{align*}
or, equivalently, $\sum_{i+j = n+1, i, j \geq 1} \widetilde{\varrho_j} \circ \widetilde{\varrho_i} = 0.$ Hence the result follows from Lemma \ref{coder-lemma}.
\end{proof}

In view of Proposition \ref{prop-dend-dend1} and Theorem \ref{thm-dend1-coder}, we get the following.

\begin{thm}\label{thm-dend-coder}
A $\mathrm{Dend}_\infty$-algebra structure on a graded vector space $A$ is equivalent to a degree $-1$ coderivation $D \in \mathrm{Coder}^{-1} (\mathrm{Diass}^c (V))$ with $D \circ D = 0$, where $V = sA$.
\end{thm}

Thus, we obtain a new interpretation of a $\text{Dend}_\infty$-algebra structure on $A$ in terms of a coderivation on the graded free diassociative coalgebra on $sA$. This interpretation allows one define cohomology and deformation theory of $\text{Dend}_\infty$-algebras.

\medskip

\subsection{Relation with $A_\infty$-algebras}
Like a dendriform algebra gives rise to an associative algebra, any $\mathrm{Dend}_\infty$-algebra gives rise to an $A_\infty$-algebra. This can be stated as follows.

\begin{thm}\label{split-a-inf}
Let $(A, \mu_{k, [r]})$ be a $\mathrm{Dend}_\infty$-algebra. Then $(A, \mu_k)$ is an $A_\infty$-algebra where
\begin{align*}
\mu_k := \mu_{k, [1]} + \mu_{k, [2]} + \cdots + \mu_{k, [k]}, ~~~ \mathrm{ for }~~~ 1 \leq k < \infty.
\end{align*}
\end{thm}

\begin{proof}
Since $(A, \mu_{k, [r]})$ is a $\mathrm{Dend}_\infty$-algebra, for any $n \geq 1$ and $1 \leq r \leq n$, we have
\begin{align*}
\sum_{i+j = n+1} \sum_{\lambda =1}^j \pm~ \mu_{j, R_0 (j ; 1, \ldots, i, \ldots, 1)[r]} \big( a_1, \ldots, a_{\lambda -1} , \mu_{i, R_\lambda (j ; 1, \ldots, i, \ldots, 1)[r]} (a_\lambda, \ldots, a_{\lambda + i -1}), a_{\lambda + i}, \ldots, a_n \big) = 0.
\end{align*}
Adding these relations for $r = 1, 2, \ldots, n$, we get
\begin{align}\label{r=1}
\sum_{i+j = n+1} \sum_{\lambda =1}^j \pm~ \bigg( \sum_{r = 1}^n~ \mu_{j, R_0 (j ; 1, \ldots, i, \ldots, 1)[r]} \big( a_1, \ldots, a_{\lambda -1} , \mu_{i, R_\lambda (j ; 1, \ldots, i, \ldots, 1)[r]} (a_\lambda, \ldots, a_{\lambda + i -1}),\\
 a_{\lambda + i}, \ldots, a_n \big) \bigg) = 0. \nonumber
\end{align}
For any fixed $i, j$ and $\lambda$, we may write
\begin{align*}
\sum_{r=1}^n = \sum_{r = 1}^{\lambda -1} + \sum_{r = \lambda}^{\lambda + i -1} + \sum_{r = \lambda + i}^n.
\end{align*}
One observes that
\begin{align}\label{a-inf-1}
&\sum_{r = 1}^{\lambda -1} ~\mu_{j, R_0 (j ; 1, \ldots, i, \ldots, 1)[r]} \big( a_1, \ldots, a_{\lambda -1} , \mu_{i, R_\lambda (j ; 1, \ldots, i, \ldots, 1)[r]} (a_\lambda, \ldots, a_{\lambda + i -1}), a_{\lambda + i}, \ldots, a_n \big) \nonumber \\
&= (\mu_{j, [1]} + \cdots + \mu_{j, [\lambda -1]}) (a_1, \ldots, a_{\lambda -1}, \mu_i (a_\lambda, \ldots, a_{\lambda + i -1}), a_{\lambda + i}, \ldots, a_n).
\end{align}
Similarly,
\begin{align}\label{a-inf-2}
&\sum_{r = \lambda}^{\lambda + i -1} \mu_{j, R_0 (j ; 1, \ldots, i, \ldots, 1)[r]} \big( a_1, \ldots, a_{\lambda -1} , \mu_{i, R_\lambda (j ; 1, \ldots, i, \ldots, 1)[r]} (a_\lambda, \ldots, a_{\lambda + i -1}), a_{\lambda + i}, \ldots, a_n \big) \nonumber\\
&= \mu_{j, [\lambda]} (a_1, \ldots, a_{\lambda -1}, (\mu_{i,[1]} + \cdots + \mu_{i, [i]}) (a_\lambda , \ldots, a_{\lambda + i -1}), a_{\lambda + i}, \ldots,  a_n ) \nonumber\\
&= \mu_{j, [\lambda]} (a_1, \ldots, a_{\lambda -1}, \mu_{i} (a_\lambda , \ldots, a_{\lambda + i -1}), a_{\lambda + i}, \ldots,  a_n )
\end{align}
and 
\begin{align}\label{a-inf-3}
&\sum_{r = \lambda + i}^n ~\mu_{j, R_0 (j ; 1, \ldots, i, \ldots, 1)[r]} \big( a_1, \ldots, a_{\lambda -1} , \mu_{i, R_\lambda (j ; 1, \ldots, i, \ldots, 1)[r]} (a_\lambda, \ldots, a_{\lambda + i -1}), a_{\lambda + i}, \ldots, a_n \big) \nonumber \\
&= (\mu_{j, [\lambda +1]} + \cdots + \mu_{j, [j]} ) (a_1, \ldots, a_{\lambda -1}, \mu_i (a_\lambda, \ldots, a_{\lambda + i -1}), a_{\lambda + i}, \ldots, a_n).
\end{align}
By adding (\ref{a-inf-1}) , (\ref{a-inf-2}) and (\ref{a-inf-3}), we get
\begin{align*}
\sum_{r = 1}^n =~ \mu_j (a_1, \ldots, a_{\lambda -1}, \mu_i (a_\lambda, \ldots, a_{\lambda + i -1}), a_{\lambda + i}, \ldots, a_n).
\end{align*}
Hence it follows from (\ref{r=1}) that $(A, \mu_k)$ is an $A_\infty$-algebra.
\end{proof}

\begin{remark}
It follows from the above theorem that a $\text{Dend}_\infty$-algebra is a special type of $A_\infty$-algebra whose $k$-ary operation splits into $k$ many operations and the $A_\infty$-identities also split into the same way. It should be mentioned that the splitting of an operad was introduced in \cite{pei-bai-guo}. They also showed that $\text{Dend}_\infty$-operad is a splitting of an $A_\infty$-operad. Hence they also get the same result as of Theorem \ref{split-a-inf} while using the $\text{Dend}_\infty$-algebra as defined in \cite{loday-val}. One may find that our proof is simple to understand. However, a relevant question in this direction could be to split the Stasheff polytopes (or associahedra) which is a nice geometric example of an $A_\infty$-operad. We hope to explore this in recent future.
\end{remark}

%\begin{remark}\label{split-rem}
%It follows from the above theorem that a $\mathrm{Dend}_\infty$-algebra is a special type of $A_\infty$-algebra whose $k$-ary operation splits into $k$ many operations and $A_\infty$-identities also split into same way. Pictorially, this can be seen as follows:
%
%\medskip
%
%\begin{center}
%\begin{tabular}{ c c|c|c|c|c|c|c} 
% \hline
% & & $A \rightarrow A$ & $A^{\otimes 2} \rightarrow A$ & $A^{\otimes 3} \rightarrow A$ & $\cdots$ & $A^{\otimes k} \rightarrow A$ & $\cdots$ \\ \hline 
% & & $\mu_{1, [1]}$ & $\mu_{2, [1]}$  & $\mu_{3, [1]}$ & $\cdots$ & $\mu_{k, [1]}$ & $\cdots$\\
% & & & $\mu_{2, [2]}$ & $\mu_{3, [2]}$ & $\cdots$ & $\mu_{k, [2]}$ & $\cdots$ \\ 
% $\mathrm{Dend}_\infty$ & & &  & $\mu_{3, [3]}$  & $\cdots$ & $\mu_{k, [3]}$ & $\cdots$ \\  
% & & &  &  & $\ddots$ & $\vdots$ & \\ 
% & & &  &  & & $\mu_{k, [k]}$& \\ 
% & & &  &  & & & $\ddots$ \\ 
% & & &  &  & & & \\ \hline
%$A_\infty$ & &  $\mu_1$ & $\mu_2$ & $\mu_3$ & $\cdots$ & $\mu_k$ & $\cdots$  \\ \hline \\
%\end{tabular}
%\end{center}
%\end{remark}

\medskip

\section{Rota-Baxter operator on $A_\infty$-algebras}
It is known that a Rota-Baxter operator on an associative algebra gives rise to a dendriform algebra in a natural way. In this section, we prove a homotopy version of this result by introducing Rota-Baxter operator on $A_\infty$-algebras.
%The aim of this section is to introduce Rota-Baxter operator on $A_\infty$-algebras and show that they naturally gives rise to $\mathrm{Dend}_\infty$-algebras.

\begin{defn}
Let $(A, \mu_k)$ be an $A_\infty$-algebra. A Rota-Baxter operator on it is given by a degree zero map $R: A \rightarrow A$ satisfying
\begin{align}\label{rota-defn}
\mu_k \big( R(a_1), \ldots, R (a_k) \big) = R \big( ~\sum_{i=1}^k \mu_k (R(a_1), R(a_2),\ldots, a_i, \ldots, R (a_k))  ~ \big), ~~~ \text{ for all } k \geq 1.
\end{align}
\end{defn}

Like classical terminology, the map $R$ is called a Rota-Baxter operator of weight zero. It follows that Rota-Baxter operators on associative algebras are basic examples.

Let $(A,m, R)$ be a Rota-Baxter associative algebra. A module over it consists of a pair $(M, R_M)$ in which $M$ is a $A$-representation and $R_M : M \rightarrow M$ is a linear map satisfying
\begin{align*}
R(a) \cdot R_M (m_1) = R_M (a \cdot R_M (m_1) + R(a) \cdot m_1) ~~ , ~~ R_M (m_1) \cdot R(a) = R_M (m_1 \cdot R(a) + R_M (m_1) \cdot a),
\end{align*}
for all $a \in A$ and $m_1 \in M$.
A morphism between two modules $(M, R_M)$ and $(N, R_N)$ is given by a morphism $d : M \rightarrow N$ of $A$-representations such that $R_N \circ d = d \circ R_M$.

Let $(A,m,  R)$ be a Rota-Baxter associative algebra and $d : (N, R_N) \rightarrow (M, R_M)$ be a morphism between modules. We consider the $2$-term chain complex $N \xrightarrow{d} A \oplus M$ with the following $A_\infty$-structure
\begin{align*}
\mu_2 ((a,m_1), (b, m_2)) =~& (m(a,b),~ a \cdot m_2 + m_1 \cdot b),\\
\mu_2 ((a, m_1), n_1) =~& a \cdot n_1, \\
\mu_2 (n_1, (a, m_1)) =~& n_1 \cdot a,\\
\mu_3 =~& 0,
\end{align*}
for $(a, m_1), (b, m_2) \in A \oplus M$ and $n_1 \in N$. It can be easily seen that the maps $\overline{R} : A \oplus M \rightarrow  A \oplus M$, ~ $(a, m_1) \mapsto (R(a), R_M (m_1))$ and $\overline{R}: N \rightarrow N$, $n \mapsto R_N (n)$ defines a Rota-Baxter operator on the above $A_\infty$-algebra.

\begin{thm}\label{rota-inf}
Let $(A, \mu_k)$ be an $A_\infty$-algebra and $R$ be a Rota-Baxter operator on it. Then $(A, \mu_{k, [r]})$ is a $\mathrm{Dend}_\infty$-algebra, where
\begin{align*}
\mu_{k, [r]} (a_1, \ldots, a_k) = \mu_k (R(a_1), R(a_2), \ldots, a_r, \ldots, R (a_k)),
\end{align*}
for all $k \geq 1$ and $1  \leq r \leq k$.
\end{thm}

\begin{proof}
Since $R$ is a Rota-Baxter operator on $(A, \mu_k)$, it follows from (\ref{rota-defn}) that
\begin{align*}
\mu_k (R(a_1), \ldots, R(a_k)) = R (\sum_{i=1}^k \mu_{k, [i]} (a_1, \ldots, a_k)).
\end{align*}
Moreover, the $A_\infty$ condition on $(A, \mu_k)$ implies that
\begin{align}\label{a-inf-identi}
\sum_{i+j = n+1}^{} \sum_{\lambda =1}^{j} \pm ~ \mu_{j} \big( a_1, \ldots,  a_{\lambda -1}, \mu_i ( a_{\lambda}, \ldots, a_{\lambda + i-1}),
 a_{\lambda + i}, \ldots, a_n   \big) = 0,
\end{align} 
for all $n \geq 1$ and $a_1, \ldots, a_n \in A$. In identity (\ref{a-inf-identi}), replace the tuple $(a_1, \ldots, a_n)$ of elements in $A$ by $(Ra_1, \ldots, a_r, \ldots, Ra_n)$, for some fixed $1 \leq r \leq n$.
For any fixed $i, j$ and $\lambda$, if $r \leq \lambda -1$, then the term inside the summation look
 \begin{align*}
& \mu_{j} \big( R  a_1, \ldots, a_r , \ldots  R a_{\lambda -1}, \mu_i ( Ra_{\lambda}, \ldots, R a_{\lambda + i-1}),
R a_{\lambda + i}, \ldots, R a_n   \big)\\
& = \mu_j ( R  a_1, \ldots, a_r , \ldots  R a_{\lambda -1} , R (\sum_{k=1}^i \mu_{i, [k]} (a_\lambda , \ldots, a_{\lambda +i -1}) ), Ra_{\lambda+i}, \ldots, Ra_n \big) \\
& = \mu_{j, [r]} (a_1, \ldots, a_{\lambda -1} , \sum_{k=1}^i \mu_{i, [k]} (a_\lambda , \ldots, a_{\lambda +i -1}) , a_{\lambda +i}, \ldots, a_n)\\
& = \mu_{j, R_0 (j; 1, \ldots, i, \ldots, 1)[r]} (a_1, \ldots, a_{\lambda -1} , \mu_{i, R_\lambda (j; 1, \ldots, i, \ldots, 1)[r]} (a_{\lambda}, \ldots, a_{\lambda+i-1}), a_{\lambda +i}, \ldots, a_n).
 \end{align*}
If $\lambda \leq r \leq \lambda +i -1$, we get
\begin{align*}
& \mu_j (Ra_1, \ldots, R a_{\lambda -1}, \mu_i (Ra_\lambda, \ldots, a_r, \ldots, R a_{\lambda -i +1}), R a_{\lambda + i}, \ldots, R a_n) \\
& = \mu_j (Ra_1, \ldots, R a_{\lambda -1}, \mu_{i, [r - (\lambda -1)]} (a_\lambda, \ldots, a_{\lambda + i-1}), R a_{\lambda + i}, \ldots, R a_n ) \\
& = \mu_{j, \lambda } (a_1, \ldots, a_{\lambda -1}, \mu_{i, [r - (\lambda -1)]} (a_\lambda, \ldots, a_{\lambda + i-1}),  a_{\lambda + i}, \ldots,  a_n ) \\
& =  \mu_{j, R_0 (j; 1, \ldots, i, \ldots, 1)[r]} (a_1, \ldots, a_{\lambda -1} , \mu_{i, R_\lambda (j; 1, \ldots, i, \ldots, 1)[r]} (a_{\lambda}, \ldots, a_{\lambda+i-1}), a_{\lambda +i}, \ldots, a_n).
 \end{align*}
 Similarly, if $\lambda + i \leq r \leq n$,
 \begin{align*}
& \mu_j (Ra_1, \ldots, R a_{\lambda -1}, \mu_i (Ra_{\lambda}, \ldots, R a_{\lambda + i -1 }), R a_{\lambda + i}, \ldots, a_r , \ldots,  Ra_n) \\
& =  \mu_{j, R_0 (j; 1, \ldots, i, \ldots, 1)[r]} (a_1, \ldots, a_{\lambda -1} , \mu_{i, R_\lambda (j; 1, \ldots, i, \ldots, 1)[r]} (a_{\lambda}, \ldots, a_{\lambda+i-1}), a_{\lambda +i}, \ldots, a_n).
 \end{align*}
 Therefore, we get the identities of a $\mathrm{Dend}_\infty$-algebra.
\end{proof}

\medskip

%\subsection{Rota-Baxter operator on $A_\infty$-algebras}

\section{$2$-term $\mathrm{Dend}_\infty$-algebras}
In this section we study $\mathrm{Dend}_\infty$-algebras whose underlying graded vector space is concentrated in degrees $0$ and $1$. We call them $2$-term $\mathrm{Dend}_\infty$-algebras. We put special emphasis on skeletal and strict $2$-term $\mathrm{Dend}_\infty$-algebras. The study of this section is analogous to the work of Baez and Crans for $L_\infty$-algebras \cite{baez-crans}. 

\begin{defn}(/{\bf Proposition.})\label{defn-2-term}
A $2$-term $\mathrm{Dend}_\infty$-algebra consists of the following data
\begin{itemize}
\item a chain complex $A_1 \xrightarrow{d} A_0,$
\item bilinear maps $\mu_2 : \mathbb{K} [C_2] \otimes (A_i \otimes A_j) \rightarrow A_{i+j},$
\item a trilinear maps $\mu_3 : \mathbb{K}[C_3] \otimes (A_0 \otimes A_0 \otimes A_0) \rightarrow A_1$
\end{itemize}
such that the following identities are hold:
\begin{itemize}
\item[(i)] $\mu_2 ([r]; m,n) = 0$,

\medskip

\item[(ii)] $d \mu_2 ([r]; a,m) = \mu_2 ([r]; a, dm)$,

\medskip

\item[(iii)] $d \mu_2 ([r]; m,a) = \mu_2 ([r]; dm, a)$,

\medskip

\item[(iv)] $\mu_2 ([r]; dm, n) = \mu_2 ([r]; m, dn)$,

\medskip

\item[(v)]$(\mu_2 \circ_1 \mu_2 - \mu_2 \circ_2 \mu_2 )([s]; a, b, c) =  d\mu_3 ([s];~a, b, c),$
%\item[(v)] $\mu_2 \big( R_0 (2;2,1)[s]; ~ \mu_2 (R_1 (2;2,1)[s]; a,b),~ c \big) - \mu_2 \big(  R_0 (2; 1, 2)[s]; ~a,~ \mu_2 (R_2 (2; 1, 2)[s]; b,c) \big) = d\mu_3 ([s];~a, b, c)$, \\

\medskip

\item[(vi1)] $(\mu_2 \circ_1 \mu_2 - \mu_2 \circ_2 \mu_2 )([s]; a, b, m) =  \mu_3 ([s];~a, b, dm),$
%\item[(vi1)] $ \mu_2 \big( R_0 (2;2,1)[s];~ \mu_2 (R_1 (2;2,1)[s]; a,b),~ m \big) - \mu_2 \big( R_0 (2; 1, 2)[s];~ a,~ \mu_2 (R_0 (2; 1, 2)[s]; b,m) \big) = \mu_3 ([s]; a, b, dm),$\\

\medskip

\item[(vi2)] $(\mu_2 \circ_1 \mu_2 - \mu_2 \circ_2 \mu_2 )([s]; a,  m, c) =  \mu_3 ([s];~a, dm, c),$
%\item[(vi2)] 	$\mu_2 \big( R_0 (2;2,1)[s];~ \mu_2 (R_1 (2;2,1)[s]; a,m), c \big) - \mu_2 \big( R_0 (2; 1, 2)[s];~ a, \mu_2 ( R_0 (2; 1, 2)[s]; m,c) \big) = \mu_3 ([s]; a, dm, c)$,\\

\medskip

\item[(vi3)]  $(\mu_2 \circ_1 \mu_2 - \mu_2 \circ_2 \mu_2 )([s]; m,  b, c) =  \mu_3 ([s];~ dm, b, c),$
%\item[(vi3)]  $\mu_2 \big( R_0 (2;2,1)[s];~ \mu_2 (R_1 (2;2,1)[s]; m,b), c \big) - \mu_2 \big( R_0 (2; 1, 2)[s];~ m, \mu_2 (R_0 (2; 1, 2)[s]; b,c) \big) = \mu_3 ([s]; dm, b, c)$,\\

\medskip

\item[(vii)] $(\mu_3 \circ_1 \mu_2 - \mu_3 \circ_2 \mu_2 + \mu_3 \circ_3 \mu_2) ([t]; a, b, c, e) =~ (\mu_2 \circ_1 \mu_3 + \mu_2 \circ_2 \mu_3) ([t]; a, b, c, e), $
%\item[(vii)] $\mu_3  \big( R_0 (3;2,1,1)[t];~ \mu_2 (R_1 (3;2,1,1)[t]; a,b), c, e \big) - \mu_3 \big( R_0 (3; 1,2,1)[t];~ a, \mu_2 (R_2 (3; 1,2,1)[t]; b,c), e \big) + \mu_3 \big( R_0 (3; 1,1,2)[t]; a, b, \mu_2 (R_3 (3; 1,1,2)[t]; c,e) \big) \\
%		 = \mu_2 \big( R_0 (2; 3, 1)[t]; ~  \mu_3 (R_1 (2; 3, 1)[t]; a,b,c), e \big) + \mu_2 \big( R_0 (2;1,3)[t]; a, \mu_3 (R_2 (2;1,3)[t]; b,c,e) \big),$ \\
	\end{itemize}
for all $a, b, c, e \in A_0; ~m, n \in A_1;~ [r] \in C_2;~ [s] \in C_3$ and $[t] \in C_4.$
\end{defn}

We denote a $2$-term $\mathrm{Dend}_\infty$-algebra by $(A_1 \xrightarrow{d} A_0, \mu_2, \mu_3).$ It follows from Remark \ref{split-rem} that a $2$-term $\mathrm{Dend}_\infty$-algebra can be thought as a splitting of a $2$-term $A_\infty$-algebra.

\medskip

\subsection{Skeletal algebras}

A $2$-term $\mathrm{Dend}_\infty$-algebra $(A_1 \xrightarrow{d} A_0, \mu_2, \mu_3)$ is called skeletal if $d = 0$.

In the next theorem, we show that skeletal algebras are closely related to the cohomology theory of dendriform algebras introduced in Section \ref{sec-2}.

\begin{thm}\label{skeletal-2}
There is a one-to-one correspondence between skeletal $\mathrm{Dend}_\infty$-algebras and tuples $(A, M, \theta_1, \theta_2, \sigma)$ consisting of a dendriform algebra $A$, a representation $(M, \theta_1, \theta_2)$ and a $3$-cocycle $\sigma : \mathbb{K}[C_3] \otimes A^{\otimes 3} \rightarrow M$ of the dendriform algebra $A$ with coefficients in $M$.
%
%Moreover, it extends to a one-to-one correspondence between equivalence classes of skeletal algebras and tuples $(A, M, \theta_1, \theta_2, [\rho] \in H^3_{dend} (A, M)).$
\end{thm}

\begin{proof}
Let $(A_1 \xrightarrow{0} A_0, \mu_2, \mu_3)$ be a skeletal $\mathrm{Dend}_\infty$-algebra. Then it follows from condition (v) in Definition \ref{defn-2-term} that $A_0$ is a dendriform algebra with the multiplication $\mu_2 : \mathbb{K}[C_2] \otimes (A_0)^{\otimes 2} \rightarrow A_0$.
 Moreover, conditions (vi1), (vi2) and (vi3) says that $A_1$ is a representation of $A_0$ with
\begin{align*}
\theta_1 =~& \mu_2 : \mathbb{K}[C_2] \otimes (A_0 \otimes A_1) \rightarrow A_1, \\
\theta_2 =~& \mu_2 : \mathbb{K}[C_2] \otimes (A_1 \otimes A_0) \rightarrow A_1.
\end{align*}
Finally, note that the condition (vii) is equivalent to
\begin{align*}
(\delta_{\mathrm{dend}} \mu_3 ) ([t]; a, b, c, e) = 0, ~~~ \mathrm{ for all }~~~~ [t] \in C_4.
\end{align*}
In other words, $\mu_3$ defines a $3$-cocycle of $A_0$ with coefficients in $A_1$.

Conversely, given a tuple $(A, M, \theta_1, \theta_2, \sigma)$ as in the statement, we define  $A_0 = A$ and $A_1 = M$. Moreover, we define
\begin{align*}
\mu_2([r]; a, b) =~& \pi_A ([r]; a, b),\\
\mu_2 ([r]; a, m) =~& \theta_1 ([r]; a, m) , \\
\mu_2 ([r]; m, a) =~& \theta_2 ([r]; m, a),\\
\mu_3 ([s]; a, b, c) =~& \sigma ([s]; a, b, c).
\end{align*}
It is easy to verify that with these definitions, $(A_1 \xrightarrow{0} A_0, \mu_2, \mu_3)$ is a skeletal $\mathrm{Dend}_\infty$-algebra.
%The last part follows from the definition of equivalence between skeletal algebras.
\end{proof}

One can generalize this theorem by considering 
 $\mathrm{Dend}_\infty$-algebras whose underlying chain complex is concentrated only in degrees $0$ and $n-1$, and zero differential, namely
\begin{align}\label{n-skel}
 (A_{n-1} \xrightarrow{0} 0 \xrightarrow{0} \cdots \xrightarrow{0} 0 \xrightarrow{0} A_0).
\end{align}
For degree reasons, it is easy to see that the only non-trivial binary maps are $\mu_2 : \mathbb{K}[C_2] \otimes (A_0 \otimes A_{0}) \rightarrow A_{0}$, $\mu_2 : \mathbb{K}[C_2] \otimes (A_0 \otimes A_{n-1}) \rightarrow A_{n-1}$ and $\mu_2 : \mathbb{K}[C_2] \otimes (A_{n-1} \otimes A_{0}) \rightarrow A_{n-1}$. Among the other maps, the map 
$\mu_{n+1} : \mathbb{K}[C_{n+1}] \otimes (A_0)^{\otimes n+1} \rightarrow A_{n-1}$ is only non-trivial.

 It follows from the definition of $\mathrm{Dend}_\infty$-algebra that $A= A_0$ is a dendriform algebra with the multiplication $\mu_2$. Moreover, the maps $\mu_2 : \mathbb{K}[C_2] \otimes (A_0 \otimes A_{n-1}) \rightarrow A_{n-1}$ and $\mu_2 : \mathbb{K}[C_2] \otimes (A_{n-1} \otimes A_0) \rightarrow A_{n-1}$ defines a representation of $A_0$ on the vector space $M= A_{n-1}$.

Moreover, we have the following.

\begin{thm}\label{skeletal-n}
There is a one-to-one correspondence between $\mathrm{Dend}_\infty$-algebras whose underlying chain complex are of the form (\ref{n-skel}) and tuples $(A, M, \theta_1, \theta_2, \sigma)$ where $A$ is a dendriform algebra,  $(M, \theta_1, \theta_2)$ is a representation of it and $\sigma : \mathbb{K}[C_{n+1}] \otimes A^{\otimes (n+1)} \rightarrow M$ is a $(n+1)$-cocycle of $A$ with coefficients in $M$.
%Moreover, it extends to a one-to-one correspondence between equivalence classes of such $n$-term $\mathrm{Dend}_\infty$-algebras and tuples $(A, M, \theta_1, \theta_2, [\rho] \in H^{n+1}_{\mathrm{dend}} (A,M)).$
\end{thm}

\begin{proof}
Let $( A_{n-1} \xrightarrow{0} 0 \xrightarrow{0} \cdots \xrightarrow{ 0} 0 \xrightarrow{0} A_0 , \mu_2, \mu_{n+1})$ be a such $\mathrm{Dend}_\infty$-algebra.
	Note that the condition (\ref{dend-inf-iden}) in the definition of a $\mathrm{Dend}_\infty$-algebra implies that
	\begin{align*}
	\sum_{i+j = n+3}^{} \sum_{\lambda =1}^{j} (-1)^{\lambda (i+1)} ~ \mu_{j, R_0 (j;1, \ldots, i, \ldots, 1)[r]} \big(  a_1, \ldots, \mu_{i, R_0 (j;1, \ldots, i, \ldots, 1)[r]} ( a_{\lambda}, \ldots, a_{\lambda + i-1}), \ldots,  a_{n+2}   \big) = 0,
	\end{align*}
	for all $[r] \in C_{n+2}$ and $a_1, \ldots, a_{n+2} \in A_0$. 
Since the non-zero terms in the above summation occurs for $(i=2,~ j = n+1)$ and $(i = n+1,~ j = 2)$, we get
	\begin{align*}
	&\sum_{\lambda = 1}^{n+1}  (-1)^{3 \lambda} ~\mu_{n+1, R_0 (n+1; 1, \ldots, 2, \ldots, 1)[r]}  \big( a_1, \ldots, a_{\lambda -1} , \mu_{2, R_\lambda (n+1; 1, \ldots, 2, \ldots, 1)[r]} (a_\lambda , a_{\lambda +1} ), a_{\lambda +2}, \ldots, a_{n+2}   \big) \\
	&+ (-1)^{n+2} \mu_{2, R_0 (2; n+1, 1)[r]}  \big(    \mu_{n+1, R_1 (2; n+1, 1)[r]} (a_1, \ldots, a_{n+1}),~ a_{n+2} \big ) \\
	&+ \mu_{2, R_0 (2;1, n+1)[r]} \big(  a_1, \mu_{n+1, R_2 (2;1, n+1)[r]} (a_2, \ldots, a_{n+2})   \big) = 0,
	\end{align*}
	for all $[r] \in C_{n+2}$ and $a_1, \ldots, a_{n+2} \in A_0$.
	This is equivalent to the fact that $\mu_{n+1} : \mathbb{K}[C_{n+1}] \otimes (A_0)^{\otimes n+1} \rightarrow A_{n-1}$ defines a $(n+1)$-cocycle of the dendriform algebra $A_0$ with coefficients in the representation $M=A_{n-1}$.
	
	The converse part is similar to Theorem \ref{skeletal-2}.
\end{proof}

\medskip

\subsection{Strict algebras}

A $2$-term $\mathrm{Dend}_\infty$-algebra $(A_1 \xrightarrow{d} A_0, \mu_2, \mu_3)$ is called strict if $\mu_3 = 0$.

\begin{exam}\label{exm-str1}
Let $A$ be a dendriform algebra with the corresponding multiplication $\pi_A$. Then $(A \xrightarrow{\text{id}} A, \mu_2 = \pi_A, \mu_3 = 0)$ is a strict $\text{Dend}_\infty$-algebra.
\end{exam}

\begin{exam}\label{exm-str2}
Let $A$ be a dendriform algebra and $N \xrightarrow{f} M$ be a morphism between $A$-representations. Then the chain complex $N \xrightarrow{f} A \oplus M$ inherits a strict $\text{Dend}_\infty$-algebra given by
\begin{align*}
\mu_2 ([r]; (a, m_1), (b, m_2)) =~& (\pi_A ([r]; a, b),~ \theta_1^M ([r]; a, m_2) + \theta_2^M ([r]; m_1, b)),\\
\mu_2 ([r]; (a, m_1), n) =~& \theta_1^N ([r]; a, n),\\
\mu_2 ([r]; n, (a, m_1)) =~& \theta_2^N ([r]; n, a),\\
\mu_3 =~& 0,
\end{align*}
for $(a, m_1), (b, m_2) \in A \oplus M$ and $n \in N$. This is called the semi-direct product of $A$ and $N \xrightarrow{f} M$. When $N=0$, we get the usual semi-direct product (cf. Proposition \ref{prop-semi}). 
\end{exam}

In the next, we show that strict algebras correspond to crossed module of dendriform algebras in the following sense.

\begin{defn}\label{defn-crossed-mod}
A crossed module of dendriform algebras consist of a quadruple $(A, B, dt, \theta_1, \theta_2)$ where $A$, $B$ are dendriform algebras, $dt : A \rightarrow B$ is a morphism of dendriform algebras and 
\begin{align*}
\theta_1 : \mathbb{K}[C_2] \otimes (B \otimes A) \rightarrow A \\
\theta_2 : \mathbb{K}[C_2] \otimes (A \otimes B) \rightarrow A
\end{align*}
defines a representation of $B$ on $A$, such that for $m, n \in A$, $b \in B$, $[r] \in C_2$ and $[s] \in C_3$,
\begin{align*}
dt ( \theta_1 ([r]; b, m)) =~& \pi_B ([r]; b, dt(m)), \\
dt ( \theta_2 ([r]; m, b)) =~& \pi_B ([r]; dt(m), b), \\
\theta_1 ([r]; dt(m), n) =~& \pi_A ([r]; m, n),\\
\theta_2 ([r]; m, dt(n)) =~& \pi_A ([r]; m, n),\\
(\theta_1 \circ_2 \pi_A ) ([s]; b, m, n) =~& (\pi_A \circ_1 \theta_1) ([s]; b, m,n),\\
(\theta_2 \circ_1 \pi_A) ([s]; m, n, b) =~& (\pi_A \circ_2 \theta_2) ([s]; m, n, b).
\end{align*}
\end{defn}

In a similar manner, a crossed module of dendriform algebras gives rise to a crossed module of associative algebras.

\begin{thm}\label{strict-cross}
There is a one-to-one correspondence between strict $\mathrm{Dend}_\infty$-algebras and crossed module of dendriform algebras.
\end{thm}

\begin{proof}
	Let $(A_1 \xrightarrow{d} A_0, \mu_2 , \mu_3 = 0)$ be a strict  $\mathrm{Dend}_\infty$-algebra. Take $B = A_0$ and consider the map $\pi_B = \mu_2 : \mathbb{K}[C_2] \otimes (A_0 \otimes A_0) \rightarrow A_0$. It follows from the property (v) of Definition \ref{defn-2-term} that $B$ is a dendriform algebra. Take $A = A_1$ and define $\pi_A : \mathbb{K}[C_2] \otimes (A_1 \otimes A_1) \rightarrow A_1$ by $\pi_A ([r]; m,n) := \mu_2 ([r]; dm, n) = \mu_2 ([r]; m, dn)$, for $m, n \in A_1$. Then
	\begin{align*}
	& (\pi_A \circ_1 \pi_A - \pi_A \circ_2 \pi_A )([s]; ~m,n, p)\\
	& = \pi_A \big( R_0 (2;2,1)[s];~ \pi_A (R_1 (2;2,1)[s]; m, n), p \big) - \pi_A \big( R_0 (2;1,2)[s]; ~m,~ \pi_A (R_2 (2;1,2)[s]; n, p) \big) \\
	& = \mu_2 \big( R_0 (2;2,1)[s];~ \mu_2 (R_1 (2;2,1)[s]; dm, dn), p \big) - \mu_2 \big( R_0 (2;1,2)[s]; ~dm,~ \mu_2 (R_2 (2;1,2)[s]; dn, p) \big) \\
	& = 0 \quad ( \text{by (vi1)} ).
	\end{align*}
	This shows that $A$ is also a  dendriform algebra.
	Next, take $dt = d : A_1 \rightarrow A_0$. Then
	\begin{center}
	$d ( \pi_A ([r]; m,n) ) = d ( \mu_2 ([r]; dm,n)) = \mu_2 ([r]; dm,dn) = \pi_B ([r]; dm, dn), ~~\mathrm{ for }~~~ m, n \in A_1 = A.$
	\end{center}
	Hence, $dt$ is a morphism of dendriform algebras. Finally, we define
	\begin{align*}
	\theta_1 : \mathbb{K}[C_2] \otimes (B \otimes A) \rightarrow A, ~ ([r]; b, m) & \mapsto \mu_2 ([r]; b, m), \\
	\theta_2 : \mathbb{K}[C_2] \otimes (A \otimes B) \rightarrow A, ~ ([r]; m, b) & \mapsto \mu_2 ([r]; m , b),  ~~~~~\mathrm{ for }~~~~ b \in B, m \in A.
	\end{align*}
	By using (vi1), (vi2) and (vi3), it is easy to see that $\theta_1, \theta_2$ defines a representation on $A$. Moreover, one can easily verify that the conditions of Definition \ref{defn-crossed-mod} hold.

	Therefore, $(A, B, dt, \theta_1, \theta_2)$ is a crossed module of dendriform algebras.
	
\medskip
	
Conversely, given a crossed module of dendriform algebras $(A, B, dt, \theta_1, \theta_2)$, one can construct a strict  $\mathrm{Dend}_\infty$-algebra as follows. Take $A_1 = A$, $A_0 = B$ and define $d:= dt.$ The structure maps $\mu_2 : \mathbb{K}[C_2] \otimes A_i \otimes A_j \rightarrow A_{i+j}$ are defined by
	\begin{center}
	$\mu_2 ([r]; b, b') = \pi_B ([r];  b, b'), \quad \mu_2 ([r]; b, m) = \phi ([r]; b, m), \quad \mu_2 ([r]; m, b) = \phi ([r];  m, b)$
	\end{center}
	\begin{center}
	and $\mu_2 ([r]; m,n) = 0,$
	\end{center}
	for $b, b' \in A_0 = B,$ and $m, n \in A_1 = A$. It is easy to verify that these structures give rise to a strict $\mathrm{Dend}_\infty$-algebra structure $(A_1 \xrightarrow{d} A_0, \mu_2, \mu_3 = 0)$.
\end{proof}

Note that the crossed module corresponding to the Example \ref{exm-str1} is given by $(A, A, \text{id}, \pi_A, \pi_A).$

\medskip

\section{Appendix}
In this appendix, we recall some basics on operad and $A_\infty$-algebras which are essential for the main contents of the paper. For more details, see \cite{gers-voro, loday-val, stas}.\\
%\subsection{}

\noindent {\bf Operad with multiplication.}

\begin{defn}
A non-symmetric operad (non-$\sum$ operad) in the category of vector spaces is given by a collection of vector spaces $\mathcal{O} = \{ \mathcal{O}(k)|~ k \geq 1 \}$ together with partial compositions
\begin{align*}
\circ_i : \mathcal{O}(m) \otimes \mathcal{O}(n) \rightarrow \mathcal{O}(m+n-1), ~~~~ 1 \leq i \leq m
\end{align*}
satisfying the following
\begin{align*}
(f \circ_i g) \circ_{i+j-1} h = f \circ_i (g \circ_j h), \quad &\mbox{~~~ for } 1 \leq i \leq m, ~1 \leq j \leq n,\\
(f \circ_i g) \circ_{j+n-1} h = (f \circ_j h) \circ_i g, \quad  & \mbox{~~~ for } 1 \leq i < j \leq m,
\end{align*}
and there is an (identity) element $\mathrm{id} \in \mathcal{O}(1)$ such that 
$ f \circ_i \mathrm{id} = f =\mathrm{id}\circ_1 f,$ for all $f \in \mathcal{O}(m)$ and $1 \leq i \leq m$.
\end{defn}

A toy example of an operad is given by the endomorphism operad $\text{End}_A$ associated to a vector space $A$. More precisely, $\text{End}_A (k) = \text{Hom}_\mathbb{K} (A^{\otimes k}, A)$, for $k \geq 1$, and partial compositions are precisely given by the Gerstenhaber's $\circ_i$ product defined in \cite{gers}.

Using partial compositions in an operad, one can define a circle product $\circ : \mathcal{O}(m) \otimes \mathcal{O}(n) \rightarrow \mathcal{O}(m+n-1)$ by
\begin{align*}
f \circ g = \sum_{i=1}^m ~(-1)^{(i-1)(n-1)}~ f \circ_i g, ~~~ f \in \mathcal{O}(m), g \in \mathcal{O}(n).
\end{align*}
The product $\circ$ is not associative, however, they satisfy certain pre-Lie identity:
\begin{align}\label{pre-lie-iden}
(f \circ g) \circ h - f \circ (g \circ h) = (-1)^{(n-1)(p-1)} ((f \circ h) \circ g - f \circ (h \circ g)),
\end{align}
for $f \in \mathcal{O}(m), ~g \in \mathcal{O}(n)$ and $h \in \mathcal{O}(p).$
Therefore, there is a degree $-1$ graded Lie bracket on $\oplus_{k \geq 1} \mathcal{O}(k)$ by
\begin{align}\label{lie-brckt}
[f,g] := f \circ g - (-1)^{(m-1)(n-1)} g \circ f.
\end{align}

\begin{defn}
Let $(\mathcal{O}, \circ_i, \mathrm{id})$ be an operad. A multiplication on it is given by an element $\pi \in \mathcal{O}(2)$ such that $\pi \circ \pi = 0$, or, equivalently,
\begin{align*}
\pi \circ_1 \pi = \pi \circ_2 \pi.
\end{align*}
\end{defn}

A multiplication $\pi$ defines a associative product $\oplus_{k \geq 1} \mathcal{O}(k)$ by $f \cdot g = \pm~ (\pi \circ_2 g) \circ_1 f$ and a degree $+1$ differential $d_\pi : \mathcal{O}(k) \rightarrow \mathcal{O}(k+1)$ by $d_\pi f = \pi \circ f - (-1)^{k-1} f \circ \pi$. Then it is shown in \cite{gers-voro} that the bracket $[~, ~]$ and the product $\cdot$ induces a Gerstenhaber algebra structure on the graded cohomology space $H^\bullet (\mathcal{O}, d_\pi).$\\

\noindent {\bf $A_\infty$-algebras.}

\begin{defn}
An $A_\infty$-algebra is a graded vector space $A = \oplus A_i$ together with a collection of multilinear maps $\{ \mu_k : A^{\otimes k} \rightarrow A |~ 1 \leq k < \infty \}$ with $\text{deg}(\mu_k) = k-2$, satisfying the following identities:
\begin{align*}
\sum_{i+j = n+1} \sum_{\lambda = 1}^j (-1)^{\lambda (i+1) + i (|a_1| + \cdots + |a_{\lambda -1 }|)}~& \mu_j \big( a_1, \ldots, a_{\lambda -1}, \mu_i (a_\lambda, \ldots, a_{\lambda + i -1}), a_{\lambda +i}, \ldots, a_n \big) = 0,
\end{align*}
for any  $a_1, \ldots, a_n \in A.$
\end{defn}

\medskip

\medskip

\noindent {\bf Acknowledgement.} The research is supported by the fellowship of Indian Institute of Technology, Kanpur (India). The author would like to thank the Institute for support.

%\mbox{ }\\

%\providecommand{\bysame}{\leavevmode\hbox to3em{\hrulefill}\thinspace}
%\providecommand{\MR}{\relax\ifhmode\unskip\space\fi MR }
% \providecommand{\MRhref}[2]{%
% \href{http://www.ams.org/mathscinet-getitem?mr=#1}{#2}
%}
%\providecommand{\href}[2]{#2}

\end{document}